\documentclass[a4paper,11pt,reqno]{amsart}

\pdfoutput=1

\usepackage{a4wide,color,array}
\usepackage{cite}
\usepackage{hyperref}
\usepackage{amsmath,amssymb,amsthm,color}
\hypersetup{linktocpage,colorlinks, citecolor={magenta}}
\usepackage[english]{babel}
\usepackage[utf8]{inputenc}

\usepackage{xspace}
\usepackage{bm}				
\usepackage{bbm,upgreek}		
\usepackage[e]{esvect}		
\usepackage{esint}			
\usepackage{etoolbox}			
\usepackage{calc}				
\usepackage{eso-pic}			
\usepackage{datetime}			

\usepackage{lmodern}
\numberwithin{equation}{section}

\usepackage{enumitem}
\setlist{nosep}
\setlist{noitemsep}

\newcommand{\Z}{\mathbb{Z}}

\newcommand{\R}{\mathbb{R}}

\newcommand{\I}{\mathcal{I}}

\newtheorem{theorem}{Theorem}
\newtheorem{proposition}{Proposition}[section]
\newtheorem{lemma}[proposition]{Lemma}
\newtheorem{corollary}[proposition]{Corollary}

\theoremstyle{plain}

\theoremstyle{definition}
\newtheorem{defi}[proposition]{Definition}

\theoremstyle{remark}
\newtheorem{remark}[proposition]{Remark}

\newcommand{\tref}[1]{Theorem~\ref{t.#1}}
\newcommand{\pref}[1]{Proposition~\ref{p.#1}}
\newcommand{\lref}[1]{Lemma~\ref{l.#1}}
\newcommand{\cref}[1]{Corollary~\ref{c.#1}}

\newcommand{\sref}[1]{Section~\ref{s.#1}}
\newcommand{\ssref}[1]{Subsection~\ref{ss.#1}}
\newcommand{\rref}[1]{Remark~\ref{r.#1}}

\newcommand{\eref}[1]{(\ref{e.#1})}

\def \supp{\mathrm{supp }} 
\def \1{\mathbf{1}} 
\def \mcl{\mathcal}
\def \mbb{\mathbb}

\def \dist{\mathrm{dist}}
\newcommand{\g}{\mathsf{g}}

\def\({\left(}
\def\){\right)}

\def\XXint#1#2#3{{\setbox0=\hbox{$#1{#2#3}{\int}$}
		\vcenter{\hbox{$#2#3$}}\kern-.5\wd0}}

\newcommand{\fluct}{\mathrm{fluct}}
\newcommand{\Fluct}{\mathrm{Fluct}}

\newcommand{\Iso}{\mathrm{Iso}}

\def\pa{\partial}

\newcommand{\E}{\mathbb{E}}
\renewcommand{\P}{\mathbb{P}}

\renewcommand{\bar}{\overline}

\renewcommand{\tilde}{\widetilde}

\makeatletter
\def\namedlabel#1#2{\begingroup
	#2%
	\def\@currentlabel{#2}%
	\phantomsection\label{#1}\endgroup
}
\makeatother

\def\cd{\mathsf{c}_\d}

\def \mcl {\mathcal}
\def \d {\mathsf{d}} 
\def \c {\mathsf{c}} 
\def \cd{\mathsf{c}_{\d}} 
\def \h{\mathsf{h}} 
\def \Zpart{\mathsf{Z}} 
\def \Kpart{\mathsf{K}} 
\def \Fenergy{\mathsf{F}} 
\def \Eenergy{\mathsf{E}}

\def \Henergy{\mathsf{H}} 
\def \harm{\mathsf{p}} 
\def \Haus {\mathsf{Haus}}

\def \pp {\mathbf}

\begin{document}
	\title[Maximum principle for the Coulomb gas]{A maximum principle for the Coulomb gas: microscopic density bounds, confinement estimates, and high temperature limits}
	\author{Eric Thoma}
	\date{January 5, 2025}
	
	\begin{abstract}
		We introduce and prove a maximum principle for a natural quantity related to the $k$-point correlation function of the classical one-component Coulomb gas. As an application, we show that the gas
		is confined to the droplet by a well-known effective potential in dimensions two and higher. We also prove new upper bounds for the particle density in the droplet that apply at any temperature. In particular, we give the first controls on the microscopic point process for high temperature Coulomb gases beyond the mean-field regime, proving that their laws are uniformly tight in the particle number $N$ for any inverse temperatures $\beta_N$. Furthermore, we prove that limit points are homogeneous mixed Poisson point processes if $\beta_N\to 0$.
	\end{abstract}
	
	\maketitle
	
	\section{Introduction}
	We study the classical one-component Coulomb gas at inverse temperature $\beta > 0$ and particle number $N$ in dimension $\d \geq 2$, which is specified by the Gibbs measure
	\begin{equation} \label{e.Pdef}
		\P_{N,\beta}^{V}(dX_N) = \frac{1}{\Zpart_{N,\beta}^{V}} e^{-\beta \Henergy_N^{V}(X_N)}dX_N, \quad X_N = (x_1,\ldots,x_N) \in (\R^\d)^N,
	\end{equation}
	where the Coulomb energy is given by
	\begin{equation} \label{e.Hdef}
		\Henergy_N^{V}(X_N) = \frac12 \sum_{\substack{i,j = 1 \\ i \ne j}}^N \g(x_i - x_j) + \sum_{i=1}^N V(x_i)
	\end{equation}
	for a confining potential $V(x) = V_N(x) := N^{2/\d}V_1(N^{-1/\d}x)$ for fixed $V_1$. The partition function $\Zpart^{V}_{N,\beta}$ is a normalizing factor, and $\g$ is the Coulomb potential given by $x \mapsto -\log |x|$ if $\d = 2$ and $x \mapsto |x|^{-d+2}$ if $\d \geq 3$. The scaling of the energy \eref{Hdef} is such the distance between neighboring particles $x_i, x_j$ is typically order $O(1)$.
	
	We are interested in estimates for the {\it $k$-point correlation function} $\rho_k = \rho_{k,N,\beta,V}$, a measure defined through the relation
	\begin{equation} \label{e.kpt.def}
	\int_{A_1 \times \cdots \times A_k} \rho_k(dy_1,\ldots,dy_k) = k!\binom{N}{k}  \P_{N,\beta}^V\(\bigcap_{i=1}^k\{x_i \in A_i\}\)
	\end{equation}
	for measurable sets $A_1,\ldots,A_k \subset \R^\d$. The quantity above is the expectation of the number of ordered $k$-tuples of particles that can be formed with the $i$th particle in $A_i$. We will always work in the case that $\rho_k$ has a Lebesgue density, which we will also denote by $\rho_k$.
	
	It is expected that $\rho_1$ should decay to $0$ rapidly outside the {\it droplet} $\Sigma_N = \Sigma_{N,V} \subset \R^\d$, which is a closed set solving a potential-theoretic obstacle problem associated to $V$. Moreover, there is a natural confining potential $\zeta_N$ associated to the obstacle problem, and the growth of $\zeta_N$ outside $\Sigma_N$ is expected to describe the corresponding decay of $\rho_1$, at least over large enough scales. The first purpose of the present article is to prove $\rho_1(x) \leq C e^{-\beta \zeta_N(x)}$. We call such an estimate a {\it confinement estimate} or a {\it localization estimate}, since it shows that most particles lie in close proximity to $\Sigma_N$. 
	
	
	The limit points (in a topology to be defined) of $\rho_k$, possibly recentered, as $N \to \infty$ are also of great interest. These limits describe infinite volume Coulomb gases, which in two dimensions generalize the Ginibre point process ($\beta = 2$, $V(x) = \frac12 |x|^2$) to general $\beta$ and potential. We prove that subsequential limits of $\rho_k$ exist for a wide range of temperatures, and furthermore in $\d=2$ that the resulting infinite volume Coulomb gases are homogeneous mixed Poisson point processes if $\beta_N \to 0$ as $N \to \infty$. The existence of $N \to \infty$ limit points of $\rho_k$ are new in the high temperature regime, and it will be a corollary of improved {\it microscopic control} of point counts.
	
		Our results will apply for a wide range of $\beta$, which we take to depend on $N$ under the restriction that $\theta_\ast := \inf_{N \geq 1} \beta_N N^{2/\d} > 2$. The potential $V_1$ will likewise be allowed to be very general; see our assumptions \eref{theta0.def}-\eref{VM.assume} below.
	
	Despite their varied nature, the results all follow from new maximum principles and mean value inequalities for quantities related to the $1$-point (and $k$-point) functions. These principles result from an improvement of the isotropic averaging method, which was introduced informally in \cite{L23} and then made into a systematic technique in \cite{T23}. They also bear some resemblance to an argument of Lieb (published with permission in \cite{NS14,RS16}). The isotropic averaging inequalities used here are tighter (i.e.\ have less loss from Jensen's inequality) than previous works.
	
	\subsection{Definitions and assumptions}
	We begin with some basic definitions and assumptions.
	
		The parameters $N$, $\beta$, and $V_1$ will be used throughout the paper. Accordingly, we will generally drop these parameters from notation except where needed to avoid confusion. For example, we let $V = V_N$. We warn the reader that our conventions and spatial scaling of the Coulomb gas differs from some of the literature.
		
		All implicit constants will be independent of $N$ and $\beta$ when not specified, but we allow constants to depend on certain norms and potential theoretic objects associated to $V_1$. We will only examine this dependence in certain simple cases in which it can be expressed through
		\begin{equation} \label{e.MxN.def}
			\mcl M_{x,N} := \| \max(\Delta V_1, 0) \|_{L^\infty\(B_{1}(N^{-1/\d}x)\)} = \| \max(\Delta V_N, 0) \|_{L^\infty(B_{N^{1/\d}}(x))},
		\end{equation}
		where $x \in \R^\d$ is the point at which we are estimating quantities of interest and $B_R(x)$ is the open ball of radius $R$ centered at $x$. In \cite{T23}, it was shown that $\mcl M_{x,N} \leq C_1$ is sufficient to show strong bounds on the probability of a point count in $B_R(x)$ larger than $C_2R^\d$ with $R  = \max(1,\beta^{-1/2})$ and $C_2 \gg \max(C_1,1)$. It is actually assumed $\sup_{x} \mcl M_{x,N} < \infty$ in \cite{T23}, but it is easy to see that the proofs only use boundedness of $\mcl M_{x,N}$ at relevant points $x$. See \tref{t23} for a precise statement.

	

	We assume throughout that
	\begin{equation} \label{e.theta0.def} \tag{A1}
	\theta_\ast := \inf_{N \geq 1} \beta_N N^{2/\d} > 2.
	\end{equation}
	Many results (i.e.\ all that do not involve an $N \to \infty$ limit) work at fixed $N$ again under the restriction that $\theta_N := \beta_N N^{-2/\d} > 2$. When $\theta_N$ is small, the gas is not expected to be localized to a ball of volume $O(N)$, so \eref{theta0.def} covers essentially all interesting temperature regimes for localization.

	We assume that the potential $V_1$ in \eref{Hdef} is in $C^2_{\mathrm{loc}}(\R^\d)$ with some basic growth assumptions:
	\begin{equation} \label{e.V.assume0} \tag{A2}
		\lim_{|x| \to \infty} V_1(x) + \g(x) = + \infty
	\end{equation}
	and
	\begin{equation} \label{e.V.assume} \tag{A3}
		\begin{cases}
			\int_{\{|x| \geq 1 \}} e^{-\frac {\theta_\ast} 2 (V_1(x) - \log |x|)} dx + 	\int_{\{|x| \geq 1\}} e^{-\theta_\ast( V_1(x) - \log |x|)} |x| \log^2 |x| dx < \infty  \quad \text{if } \d = 2, \\
			\int_{\{|x| \geq 1\}} e^{-\frac {\theta_\ast} 2 V_1(x)} dx < \infty \quad \text{if } \d \geq 3.
		\end{cases}
	\end{equation}
	Under \eref{V.assume0} and \eref{V.assume}, the partition function $\Zpart^{V}_{N,\beta}$ is finite.
	
	
	For low enough temperature, the {\it  empirical measure} $\frac{1}{N} \sum_{i=1}^N \delta_{N^{-1/\d }x_i}$ is well-approximated by the {\it equilibrium measure} $\mu_{\infty,1}$, which is characterized within probability measures on $\R^\d$ through the existence of a constant $c_{\infty,1}$ such that
	\begin{equation} \label{e.muinfty.def}
		\zeta_1(x) := \int_{\R^\d} \g(x-y) \mu_{\infty,1}(dy) + V_1 - c_{\infty,1}
	\end{equation}
	satisfies $\zeta_1 \geq 0$ and $\zeta_1 = 0$ on $\supp \ \mu_{\infty,1}$ except on a set of capacity $0$ (i.e.\ quasi-everywhere).
	We will almost exclusively work with a rescaling $\mu_\infty = \mu_{\infty,N}$ defined by 
	\begin{equation} \label{e.muthetaN.def}
		\mu_{\infty,N}(A) = N \mu_{\infty,1}(N^{-1/\d} A )
	\end{equation}
	 for measurable $A \subset \R^\d$. Note that $\mu_\infty$ has total mass $N$. The {\it droplet} $\Sigma = \Sigma_{N}$ is defined as the support of $\mu_{\infty}$. It is compact under \eref{V.assume0} and \eref{V.assume}, and clearly $\Sigma_N = N^{1/\d} \Sigma_1$.
	
	We define the {\it effective confining potential} $\zeta = \zeta_N$ via 
	\begin{equation} \label{e.zeta.def}
		\zeta_N(x) = N^{2/\d} \zeta_1(N^{-1/\d} x)
	\end{equation}
	where $\zeta_1$ was introduced in \eref{muinfty.def}. Note that $\zeta \geq 0$ and $\zeta = 0$ on $\Sigma$ q.e., and
	\begin{equation} \label{e.zeta.prop}
		\zeta(x) = \int_{\R^\d} \g(x-y) \mu_\infty(dy) + V - c_{\infty,N}
	\end{equation}
	for a constant $c_{\infty,N} = N^{2/\d}c_{\infty,1} + \frac12 \log N \1_{\d=2}$.
	
	For our confinement estimates, we will also assume the existence of a constant $\alpha > 0$ such that
	\begin{align}
 \label{e.droplet.assume} \tag{A4}
	&\pa \Sigma_1 \in C^{1,1}, \\
 \label{e.DValpha.assume} \tag{A5}
		&\Delta V_1(x) \geq \alpha > 0 \quad \forall x \text{ in a neighborhood of } \Sigma_1, \\
 \label{e.zeta.assume} \tag{A6}
	&\zeta_1(x) \geq \alpha\min( \dist(x,\Sigma_1)^2,1) \quad \forall x \in \R^\d.
		\end{align}
	We also assume
	\begin{equation}
	 \label{e.VM.assume} \tag{A7}
	\lim_{|x| \to \infty} \frac{V_1(x)}{\mcl M_{x,1}}  = +\infty \quad \text{if } \d \geq 3.
	\end{equation}
	With the exception of \eref{VM.assume}, these assumptions are needed to use results of \cite{AS19} on the {\it thermal equilibrium measure} in the proof of \pref{farfield.simple}. We note that \eref{droplet.assume} and \eref{DValpha.assume} imply \eref{zeta.assume} for $x$ in a neighborhood of $\Sigma_1$ by the regularity theory for the related obstacle problem \cite{C98}. So in effect \eref{zeta.assume} is an assumption that outside of any fixed neighborhood of $\Sigma_1$, the potential $\zeta_1$ stays some  distance away from $0$. Up to changing $\alpha$, it is sufficient for \eref{zeta.assume} that \eref{droplet.assume} and \eref{DValpha.assume} holds and $V_1$ is strictly convex.
	
	We define a dimensional constant $\cd$ through the relation
	\begin{equation} \label{e.cd.def}
	-\Delta \g = \cd \delta_0,
	\end{equation}
	so that $\frac{1}{\cd} \g$ is the fundamental solution of $-\Delta$. Applying the Laplacian operator to \eref{zeta.prop} within the interior of $\Sigma$ formally shows $\mu_\infty = \cd^{-1} \Delta V \1_{\Sigma}$. We also will use the electric potential generated by a measure $\nu$, defined by
	\begin{equation} \label{e.epot.def}
	\h^\nu(x) = \int_{\R^\d} \g(y-x) \nu(dy)
	\end{equation}
	whenever $\nu$ has bounded negative part in $\d \geq 3$ or bounded negative part and sufficient decay at $\infty$ in $\d  = 2$.

	We define the microscopic point process 
	\begin{equation}\label{e.mpp.def}
	\pp X_N = \sum_{i=1}^N \delta_{x_i}
	\end{equation}
	 associated to $X_N$, as well as its translates $\tau_{x^\ast} \pp X_N = \sum_{i=1}^N \delta_{x_i - x^\ast}$ for $x^\ast \in \R^\d$. We consider $\pp X_N$ as a random variable taking values in the set of locally finite, integer valued Radon measures endowed with the weak topology induced via integration against compactly supported, continuous functions. See \cite[Chapters 9 and 11]{DVJ08} for relevant background material on point processes. Uniform tightness of the laws of point processes $\pp X_N$ in this space is equivalent to uniform tightness of $\pp X_N(A)$ for any bounded Borel set $A$ \cite[Theorem 11.1.VI]{DVJ08}.
	
	\subsection{Main results and comparisons to literature}
	Our main results are consequences of the following inequality and its generalizations.
	\begin{theorem} \label{t.subharmonic}
		One has
		\begin{equation} \label{e.1pt.Iso.diff}
		-\Delta ( e^{\beta \zeta} \rho_1) \leq \beta \cd e^{\beta \zeta} \rho_1 \mu_\infty
		\end{equation}
		in an integral sense. In particular, the quantity $e^{\beta \zeta} \rho_1$ is subharmonic on the complement of $\Sigma = \supp \ \mu_\infty$.
	\end{theorem}
	This theorem is proved in \pref{1pt.Iso} (in an equivalent integral formulation alongside a closely related inequality) through a short isotropic averaging argument. Crucially, it also holds with $x \mapsto \rho_1(x)$ replaced by the conditional $1$-point function $x \mapsto \rho_1(x | y_2,\ldots,y_k)$ satisfying $\rho_k(y_1,\ldots,y_k) = \rho_{k-1}(y_2,\ldots,y_k) \rho_1(y_1 | y_2,\ldots,y_k)$ a.e.\ with respect to the $k$-point marginal of the Coulomb gas.
	
We emphasize that \eref{1pt.Iso.diff} is valid at fixed $N$, i.e.\ there is no error term and it is not an asymptotic inequality. It can be compared to the maximum principle for weighted polynomials; see \rref{weightedpoly} for a discussion. Unlike its weighted polynomial analog, \tref{subharmonic} is fundamentally stochastic, works in dimensions $\d \geq 3$ in addition to $\d = 2$, and can be applied in a mean value inequality form.

In \cite[proof of Theorem 2]{T23}, it is proved (though never explicitly stated) that
$$
-\Delta \rho_1 \leq \beta  \rho_1 \Delta V,
$$
which coincides with \eref{1pt.Iso.diff} on the interior of $\Sigma$ where $\zeta = 0$ and $\mu_\infty = \cd^{-1} \Delta V$. The improved form of \tref{subharmonic}, which comes from a refinement of the isotropic averaging technique, gives much more useful information on the complement of $\Sigma$.

Indeed, the maximum principle for $e^{\beta \zeta} \rho_1$ on $\R^\d \setminus \Sigma$ allows us to prove global upper bounds for $e^{\beta \zeta} \rho_1$ once we bound the quantity within $\Sigma$ and at $\infty$. The following result takes care of the bound within $\Sigma$.

	\begin{theorem} \label{t.mgf}
	We have for any $r > 0$, $\gamma > 0$, and $x \in \R^\d$ that
	\begin{equation}
		\E\left[ e^{\gamma \pp X_N(B_r(x))} \right]\leq e^{C^{1+\beta} \gamma r^\d}
	\end{equation}
	for a constant $C$ uniformly bounded for bounded $\mcl M_x$ and $\gamma$ (and independent of $N ,r,\beta$).
\end{theorem}

\tref{mgf} in particular implies that $\| \rho_1 \|_{L^\infty(\R^\d)} \leq C^{1+\beta}$. It is the first uniform control of high temperature, microscopic ($\beta \ll 1$, $r = O(1)$) particle densities. Previous best controls of point counts in \cite{AS19, S22,T23} only show bounded particle densities at length scales $r = \beta^{-1/2}$ (or sometimes $\beta^{-1/2} \max(\log \beta^{-1},1)^{1/2}$ in $\d =2$), where they show subgaussian probability tails at large density. It is expected that below this length scale $\pp X_N$ is governed by Poissonian statistics. The weaker exponential tails given by \tref{mgf} is consistent with this expectation.

In \pref{farfield.simple}, we prove bounds on the one-point function near $\infty$, and as a result we deduce from \tref{subharmonic} and \tref{mgf} the following confinement estimate.
\begin{theorem} \label{t.rho1bd}
	If $\d = 2$, one has
	\begin{equation} \label{e.rho1bd}
		\rho_1 \leq C^{1+\beta} e^{-\beta \zeta},
	\end{equation}
	for a constant $C$ dependent only on $\sup_{x \in \Sigma} \mcl M_x$. If $\d \geq 3$, there are $V_1$ dependent constants $C < \infty$ and $c > 0$ such that for all $N \geq C$ we have
	\begin{equation} \label{e.rho1bd3}
		\rho_1(x) \leq C^{1+\beta} \(1 + \max(-\log \beta,1) e^{-c \beta N^{2/\d}} p(x,\Sigma) \) e^{-\beta \zeta(x)}  \quad \forall x \in \R^\d
	\end{equation}
	where $p(x,\Sigma)$ is the probability that Brownian motion started at $x$ never hits $\Sigma$.
\end{theorem}

By applying a union bound to \tref{rho1bd}, we can give a result on the onset of the vacuum from the droplet edge.
\begin{corollary} \label{c.vacuum}
	 We consider any $\d \geq 2$, but if $\d \geq 3$ we also assume that $\theta_N = \beta_NN^{2/\d} \geq c_1^{-1} \log \log N$ for a small enough $c_1 > 0$ and $N$ is large enough. Then we have
	 \begin{equation} \label{e.vacuum.hard}
	 	\P\(\max_{i \in \{1,\ldots,N\}} \zeta(x_i) \geq \gamma\) \leq C^{1+\beta} \int_{\{\zeta \geq \gamma\}} e^{-\beta \zeta(x)}dx \quad \forall \gamma \geq 0.
	 \end{equation}
	
	If we further assume that \eref{V.assume} holds with $4$ in place of $\theta_\ast$, then there exists $c_V > 0$ such that, uniformly in $N$ and for all $0 \leq \gamma \leq N^{1/\d} \sqrt{\frac{\beta}{\log N}}$, we have
	\begin{equation} \label{e.vacuum.soft}
	\P\(\max_{i \in \{1,\ldots,N\}} \dist(x_i, \Sigma) \geq \gamma \sqrt{\frac{{\log N}}{\beta}}\) \leq C^{1+\beta} N^{1 - c_V\gamma^2}.
	\end{equation}
\end{corollary}

\begin{remark}
	For an easier comparison with the literature, we now write \tref{rho1bd} and \cref{vacuum} in the macroscopic coordinates $Y_N = N^{-1/\d} X_N$. The first corresponding bound is
	$$
	\bar \rho_1(y) \leq C^{1+\beta} N\(1 + \1_{\d \geq 3} \max(-\log \beta,1) e^{-c \beta N^{2/\d}} p(y,\Sigma_1)\) e^{-\beta N^{2/\d} \zeta_1(y)}
	$$
	for $\bar \rho_1(y) = N\rho_1(N^{1/\d}y)$ the $1$-point function density of $Y_N$, i.e.\  $N$ times the marginal probability density of $y_1$. The effective confining potential $\zeta_1$ is written in macroscopic coordinates as in \eref{muinfty.def}. Corresponding to \cref{vacuum}, we have for all $N$ large enough and $\beta \leq \beta_0$ fixed that
	\begin{equation} \label{e.vacuum.macro}
		\P\(\max_{i \in \{1,\ldots,N\}} \dist(y_i, \Sigma_1) \geq \frac{c_V^{-1/2}}{N^{1/\d}}\sqrt{C + \frac{\log N  + t}{\beta}} \) \leq e^{-t} \quad \forall t \leq \log N
	\end{equation}
	so long as $\beta_N \gg N^{-2/\d}\log N$ as $N \to \infty$, where $c_V$ is the constant from \eref{vacuum.soft} and $C > 0$ depends on $V$.
\end{remark}

Our \tref{rho1bd} and \cref{vacuum} have a direct precedent in the literature. In \cite[Theorem 1]{A21}, Ameur works in $\d=2$ and proves (under a similarly general but not identical set of assumptions)
\begin{equation} \label{e.ameur1}
\rho_1 \leq C^{\beta} N e^{-\beta \zeta}
\end{equation}
for any $\beta > 0$ and for a $V_1$ dependent constant $C$. Actually, Ameur presents a modified version of \eref{ameur1} by applying growth estimates for $\zeta$ near the droplet boundary. \tref{rho1bd} strengthens \eref{ameur1} by removing the factor of $N$ and extending the result to $\d \geq 3$. \cite[Theorem 2]{A21} also proves a version of \cref{vacuum} in $\d=2$ which effectively matches \eref{vacuum.macro}, so our contribution is mainly the $\d \geq 3$ extension. Finally, \cite{AT24} examines the case in which $V_1$ is a Hele-Shaw potential and $\d=2$ and combines isotropic averaging and weighted polynomial techniques to give a result similar to \eref{ameur1} but with constants with explicit potential dependence. It also proves an interesting Lipschitz estimate for $\rho_1$.

Working in all $\d \geq 2$, the bounds of \cite[Theorem 1.12]{CHM18}, which work with the inverse temperature parameter $\tilde{\beta} = N^{1-2/\d} \beta$, give that $\{x_i\}_{i=1}^N$ is contained within a ball of radius $RN^{1\d}$ with probability $e^{-cNV_1(R)}$ for all $R \geq R_0$ with $V_1$ and $\tilde \beta$ dependent constants $c$ and $R_0$, though we expect their proof can be used to achieve a similar bound as we give in \lref{farfield}, as both follow similar ideas. In the proof of \tref{rho1bd}, we will need a similar bound in $\d \geq 3$ except with the (expected) optimal constant $c$ for large $R$, which takes new techniques. This is the only explicit previous result on confinement for $\d \geq 3$ beyond classical estimates. Some previous estimates on linear statistics give that $N^{-1} \sum_{i=1}^N \delta_{N^{-1/\d} x_i}$ is close to $\mu_{\infty,1}$ in certain senses, but none are sensitive enough to single particles to effectively estimate $\max_i \dist(x_i, \Sigma)$ and none apply on microscopic scales near $\pa \Sigma$. 

In the case of quadratic $V_1(x) = \frac12 |x|^2$, $\d=2$, and $\beta = 2$, the Coulomb gas corresponds to the Ginibre ensemble from random matrix theory, and much more is known about the gas confinement due to determinantal structure. In particular, \cite{R03} proves that we have a precise approximation in law:
$$
\max_{i=1,\ldots,N} |x_i| \approx \sqrt{N} + \frac{\sqrt{\log N - 2 \log \log N - \log 2 \pi}}{2} + \frac{1}{2 \sqrt{\log N - 2 \log \log N - \log 2 \pi}} G
$$
as $N \to \infty$ for a standard Gumbel random variable $G$. Since $\Sigma = \overline{B_{\sqrt N}(0)}$, this estimate also gives the approximate law of the maximal distance from the droplet. Similarly detailed results are given with a certain class of radially symmetric potentials, but still with $\d = \beta = 2$, in \cite{CP14}.

Our result allows us to obtain a less detailed upper bound for $V_1(x) = \frac12 |x|^2$, $\d=2$, and general $\beta$ by working directly from \eref{vacuum.hard} with the explicit form of $\zeta_1(x) = V_1(x) - \log |x| - 1/2$. We omit the proof as it is simply bounding \eref{vacuum.hard} using a first order Taylor approximation to $\zeta(x)$ just outside $\Sigma$. The result below gives the expected order of $\max_{i} |x_i|$ at $\beta = 2$ up to the submicroscopic length scale $\sqrt{\frac{\log \log N}{\log N}}$, and similar calculations can easily give results for $\d \geq 3$ or small $\beta$.
\begin{corollary} \label{c.radial}
	Let $\d=2$ and $V_1(x) = \frac12 |x|^2$, and let $\beta > 0$ be fixed with $N$. Then we have
	\begin{equation} \label{e.radial}
 	\P\( \max_{i=1,\ldots,N} |x_i| \geq  \sqrt{N} + \sqrt{\frac{\log N - \log \log N + 2t}{2 \beta}}  \)\leq Ce^{-t}
	\end{equation}
	for all $N$ large enough and $t \leq \log N$. Here $C$ depends on $\beta$.
\end{corollary}

\begin{remark}
	To obtain the correct coefficient for the $\sqrt{\log \log N}$ term in \eref{radial}, the inequality \eref{vacuum.hard} is insufficient. One possible explanation is that actually $\rho_1$ satisfies a slightly stronger inequality than \eref{rho1bd} in this case, which heuristic arguments similar to the proof of \lref{farfield} suggest.
\end{remark}

\tref{mgf} implies the law of (recentered versions of) $\pp X_N$, $N \geq 1$, are uniformly tight in $N$. Such a result was previously only known for $\beta$ uniformly bounded away from $0$ \cite{AS21,T23}. \tref{subharmonic} (and its conditional one-point function version) implies that the $k$-point function of any limit point $\pp X$ is subharmonic in each variable if $\beta \to 0$ as $N \to \infty$. Since the limiting $k$-point function is bounded, it is therefore constant if $\d = 2$. Some further arguments allow us to prove the following.

\begin{theorem} \label{t.poisson2}
	Let $\d=2$, and let $\pp X$ be a weak limit point of the recentered point process 
	$$\tau_{x_N^\ast} \pp X_N = \sum_{i=1}^N \delta_{x_i - x_N^\ast}$$
	under $\P^V_{N,\beta}$, $N \geq 1$, for a sequence $x_N^\ast$ with $\sup_N N^{-1/2} |x_N^\ast| < \infty$ and $\beta_N \to 0$. Then the law of $M_R := (\pi R^2)^{-1} \pp X(B_R(0))$ has a weak limit $\mu$ as $R \to \infty$, and $\pp X$ is a mixed Poisson point process of homogeneous intensity $m$ with $m \sim \mu$. In particular, if there is a sequence $R_n \to \infty$ with
	\begin{equation} \label{e.plim.assume}
		\lim_{n \to \infty} \frac{\pp X(B_{R_n}(0))}{\pi R_n^2}= m_0
	\end{equation} 
	in probability, for some constant $m_0 \geq 0$, then $\pp X$ is a homogeneous Poisson point process of intensity $m_0$.
\end{theorem}

There are several previous results for limit points of microscopic point processes or related objects of (general) Coulomb gases. The first category of results involves ``averaging" the point process $\pp X_N$ over centering points $x_N^\ast$ in a cube much larger than $\rho_{\beta}$ in length, where $\rho_\beta \geq \beta^{-1/2}$ is a certain rigidity length scale, forming the ``(tagged) empirical process." These empirical processes satisfy a large deviation principle with a rate function of the form
\begin{equation} \label{e.LDP}
 \pp X \mapsto \mathrm{RelEnt}(\pp X) + \frac{\beta}{2} \mathrm W_{\mathrm{Coul}}(\pp X) - \mathrm{constant}
\end{equation}
where $\mathrm{RelEnt}$ is the entropy relative to a homogeneous Poisson process of an explicit intensity and $\mathrm W_{\mathrm{Coul}}$ is a jellium renormalized Coulomb energy. This result was proved for fixed $\beta$ (and including certain Riesz gases) in \cite{LS17}, and extended to include $\beta_N \to 0$ in \cite{AS21} (see also an intermediate work \cite{L17} ). In particular, it is known that the empirical process becomes Poissonian in the high temperature limit, which is not sufficient to conclude that the microscopic point process is asymptotically Poissonian for any specific centering points. It is also known that if one takes a $\beta \to 0$ limit of minimizers of \eref{LDP}, then one obtains a Poisson process \cite{L16}. See also \cite{PG23} for a similar result in the high temperature (mean field) regime with more general interactions $\g$.

In two recent breakthroughs \cite{L23, L24}, Lebl\'e proved that if $\d =2$ and $\beta > 0$ is fixed, then limit points of the microscopic point processes $\pp X_N$ are hyperuniform, satisfy canonical DLR equations, and are translation invariant. It is still an open question whether limit points are unique up to their average particle density, though of course \tref{poisson2} answers this question in the infinite temperature case, as well as proving translation invariance and DLR formalism (which are trivial for mixed Poisson processes). By comparison to the infinite temperature case, limit points with $\inf_N \beta_N > 0$ are rich in structure and require delicate techniques to study.

A unique feature of \tref{poisson2} is that applies to $\beta_N \to 0$ at arbitrarily slow speeds and without averaging over centering points. Furthermore, the powerful estimates of \cite{S22}, critical for the analysis in \cite{L23,L24}, are unavailable on microscopic scales in this regime. Remarkably, the proof of \tref{poisson2} is quite short and almost entirely avoids delicate computations, instead relying on the classical Liouville's theorem to do the heavy lifting (which also restricts it to $\d = 2$, though we can give significant information in $\d \geq 3$ as well).

If we instead demand that $\beta_N \to 0$ at a fast enough rate, specifically the ``mean-field" regime $N\beta_N^{2/\d}  \to \gamma$ for some $\gamma \geq 0$, then Lambert proves that $\pp X_N$ converges as $N \to \infty$ to a Poisson process of explicit intensity for a wide range of interactions $\g$ and any $\d \geq 1$, including the Coulomb kernel in $\d \geq 2$ \cite{Lam21}. The author also analyzes the edge behavior of the Coulomb gas (and others) in $\d \geq 3$ for certain radial potentials, proving convergence to inhomegeneous Poisson processes when the centering point is on the boundary of $\Sigma$, which in particular provides very precise versions of \cref{vacuum} in this high temperature case.

The fast decay of $\beta_N$ allows the proofs of \cite{Lam21} to treat interaction terms that appear in computations of $\rho_k$ as their macroscopic space averages. Following a related proof structure, \cite{PPT24} is able to prove converges to a Poisson processe for the wider range of temperatures $N^{-1} \ll \beta_N \ll N^{-1/2}$, e.g.\ beyond the mean field regime, but with a restricted class of $\g$ that does not include the Coulomb kernel. When $\beta_N \to 0$ at a very slow speed, we are in a tricky spot where $\beta_N$ is not small enough to use techniques from \cite{Lam21}, for which factors of $\beta$ mollify error terms, nor large enough to use estimates from \cite{S22}, which apply above a rigidity length scale $\geq \beta^{-1/2}$.

Finally, for the two-component Coulomb gas, in which $\d = 2$ and there are two species of oppositely charged particles, there has been progress on the empirical and microscopic point processes, which exhibit interesting phenomenon such as dipole formation. In particular, \cite{LSZ17} proves an LDP for the empirical fields after averaging over the macroscopic domain when $\beta < 2$ (bounded away from $0$), and recently \cite{BS24} examined the interesting process of dipole formation for general $\beta$.

\subsection{Paper organization and future directions}
The technical core of the paper is in \sref{max}, where we prove several quite general inequalities using the isotropic averaging method, including \tref{subharmonic}. In \sref{micro}, we apply these methods to prove the particle density upper bound \tref{mgf} and the high temperature result \tref{poisson2}.  \sref{confinement} proves our confinement estimates \tref{rho1bd} and \cref{vacuum}. As mentioned previously, with the maximum principle from \tref{subharmonic} and the density bound within $\Sigma$ from \tref{mgf}, the main task of \sref{confinement} becomes proving an estimate for $\limsup_{|x| \to \infty} e^{\beta \zeta} \rho_1(x)$, which is done in \pref{farfield.simple}. The proof of that proposition is relatively straightforward in $\d = 2$, but takes a novel ``squeezing" argument in $\d \geq 3$ which relies on delicate isotropic averaging estimates for particle repulsion; see \ssref{confine.extreme3} for this new technique.

Isotropic averaging techniques, and their consequences, have now found a handful of applications, for which they are flexible and lead to relatively short proofs. We expect results of this paper to be useful in future research, particularly we believe the extension of particle density bounds and $k$-point function bounds to the high temperature regime will be useful technical tools in future works.

We highlight the potential usefulness of the ``squeezing" argument in \ssref{confine.extreme3}, which is somewhat reminiscent of the ``mimicry" technique used in \cite{T23, T23_2}. A key advantage is that this new argument leads to transport costs that recombine with the Coulomb energy \eref{Hdef}. These costs can therefore be estimated by ratios of partition functions with varying temperature, for which results in the literature are readily available. We expect one could go further with this technique by for instance estimating the distribution of electric potential energy of individual particles in the gas. The maximum of the electric potential energy has recently received attention in \cite{LamLZ24, P24}.

Lastly, our results leave many interesting related open problems. For our confinement estimate, the microscopic behavior of $\rho_1$ near the boundary of the gas is very interesting, as it is expected to exhibit oscillations and a squeezing effect; see \cite{CSA20,CSA24}. Our results fall short of describing these phenomena besides for establishing decay of $\rho_1$ at microscopic distances from the droplet edge; this strong confinement effect at the edge is expected to result in a boundary layer of frozen particles. It would also be interesting to have corresponding lower bounds for \cref{vacuum}.

For \tref{poisson2}, we expect that all infinite temperature limit points of $\tau_{x_N^\ast} \pp X_N$ are Poisson processes of constant intensity $\mu_{\theta_\infty,1}(x_\infty^\ast)$ when $N^{2/\d} \beta_N \to \theta_\infty \in (2,\infty]$ and $x_\infty^\ast = \lim_{N \to \infty} N^{-1/\d} x_N^\ast$ (see \ssref{confine.thermal} for the definition of $\mu_{\theta,1}$), and we expect a similar result to also hold in $\d \geq 3$. In $\d=2$, the remaining task is to estimate $\pp X_N(B_R)/R^2$ for $R$ a fixed large number as $N \to \infty$ and $\beta_N \to 0$. This task is accomplished at $R = R_N = \beta_N^{-1/2} |\log \beta_N|^{1/2}$ in \cite{AS21}, and if similar results were known at any sequence $R_N \ll \beta_N^{-1/2}$, then a quantitative form of Liouville's theorem and \tref{subharmonic} would be sufficient to conclude.

	\subsubsection*{Acknowledgments.}\ 
The author was partially supported by NSF grant DMS-2303318.
	
	
	\section{A maximum principle for the Coulomb gas} \label{s.max}
	
	We begin with the {\it electric splitting formula}, which is a rewriting of $\Henergy^V_N(X_N)$ which brings the equilibrium measure to the forefront. It uses the jellium energy $\Fenergy$ which we first define.
	\begin{defi}
		The jellium energy $\Fenergy$ of points $X_N$ on a background measure $\mu$ of bounded Lebesgue density is defined by
		\begin{equation} \label{e.Fenergy.def}
			\Fenergy(X_N,\mu) = \frac12 \iint_{\R^\d \times \R^\d \setminus \Delta}\g(x-y) \(\sum_{i=1}^N \delta_{x_i} - \mu\)^{\otimes 2}(dx,dy),
		\end{equation}
	where  $\Delta = \{(x,x) : x \in \R^\d\}$.
	\end{defi}
	
	\begin{proposition} \label{p.split}
		Recall the (scaled) equilibrium measure $\mu_\infty = \mu_{\infty,N}$ and effective confining potential $\zeta$. We have
		\begin{equation} \label{e.split}
			\Henergy_{N}^V(X_N) = \Eenergy(\mu_\infty,V) + \Fenergy(X_N,\mu_\infty) + \sum_{i=1}^N \zeta(x_i)
		\end{equation}
		where $\Eenergy(\mu,V)  = \frac12 \iint_{\R^\d \times \R^\d} \g(x-y) \mu^{\otimes 2}(dx,dy) + \int_{\R^d} V(x) \mu(dx)$ is the electrostatic self-energy of $\mu$.
		 In particular, we have
		$$
		\P(dX_N) \propto e^{-\beta \Fenergy(X_N,\mu_\infty) - \beta \sum_{i=1}^N \zeta(x_i) } dX_N.
		$$
	\end{proposition}
	\begin{proof}
		The proof is standard (e.g.\ \cite[Lemma 5.1]{S24}), but we repeat it in our notation. By expanding $(\sum_{i} \delta_{x_i} - \mu_\infty)^{\otimes 2}$, one has
		\begin{align} \label{e.split.1}
			\Fenergy(X_N,\mu_\infty) &= \Henergy_N^V(X_N) + \frac12 \iint_{\R^\d \times \R^\d} \g(x-y) \mu_\infty^{\otimes 2}(dx,dy) - \sum_{i=1}^N V(x_i) - \sum_{i=1}^N \int_{\R^\d} \g(x_i-y) \mu_\infty(dy) \\ \notag
			&= \Henergy_N^V(X_N)  - \Eenergy(\mu_\infty,V) \\ \notag
			&\quad -\int_{\R^\d}V(x)\(\sum_{i=1}^N \delta_{x_i} - \mu_\infty\)(dx) - \int_{\R^\d} \g(x-y) \(\sum_{i=1}^N \delta_{x_i} - \mu_\infty\)(dx)  \mu_\infty(dy).
		\end{align}
		By \eref{zeta.prop}, there is a constant $c$ such that
		\begin{equation*}
			\int_{\R^\d} \g(x_i-y) \mu_\infty(dy) + V - c = \zeta(x_i).
		\end{equation*}
		It follows that the two terms in the last line of  \eref{split.1} can be rewritten as
		$$
		-\int_{\R^\d} (\zeta(x) + c)\(\sum_{i=1}^N \delta_{x_i} - \mu_\infty\)(dx) = -\sum_{i=1}^N \zeta(x_i),
		$$
		where the last equality used that $\zeta = 0$ q.e.\ within $\Sigma = \supp \ \mu_\infty$ and $\mu_\infty(\R^\d) = N$. Rearranging terms completes the proof.
	\end{proof}
	
	We now define {\it isotropic averaging} operators and examine their effect on the energy $\Fenergy(\cdot,\mu_\infty)$.

	\begin{defi}
		For an open, bounded set $\omega \subset \R^\d$, we let $\Iso_\omega : C^0(\pa \omega) \to C^0(\overline \omega)$ be the solution of the boundary value problem
		\begin{equation} \label{e.Iso.BVP}
			\begin{cases}
				-\Delta (\Iso_\omega f)(x) = 0 \quad &x \in \omega \\
				\Iso_\omega f(x) = f(x) \quad &x \in \pa \omega.
			\end{cases}
		\end{equation}
		We assume that $\omega$ is regular enough that the solution to the above problem is unique. We extend $\Iso_\omega$ to measurable functions with continuous negative parts by monotonicity.
		
		We furthermore assume that $\omega$ has a harmonic measure $\harm_\omega(x,\cdot)$, $x \in \omega$, that is absolutely continuous with respect to $(d-1)$-dimensional Hausdorff measure $\Haus_{\pa \omega}$ restricted to $\pa \omega$, and we denote its density also by $\harm_\omega(x,\cdot)$. We can then write  $$(\Iso_\omega f)(x) = \int_{\pa \omega} f(y) \harm_\omega(x,dy).$$ By convention, we take $\harm_\omega(x,\cdot)$ to be the Dirac mass at $x$ if $x \not \in \omega$. In this way, we consider $\Iso_\omega$ to act on functions of $X_N$ that have continuous negative parts by specifying the index of the variable on which it operates, i.e.
		$$
		\Iso_{\omega,i} f(X_N) = \int_{\pa \omega} f(x_1,\ldots,x_{i-1},y,x_{i+1},\ldots,x_N) \harm_\omega(x_i,dy).
		$$
		Note that $\Iso_{\omega,i} f(X_N) = f(X_N)$ for $x_i \not \in \omega$ by our convention.
		For index sets $\I \subset \{1,\ldots,N\}$, we let $\Iso_{\omega,\I} = \otimes_{i \in \I} \Iso_{\omega,i}$, which is also, by the above notation convention, the functional composition of each $\Iso_{\omega,i}$, $i \in \I$, in any order.
	\end{defi}

	\begin{proposition} \label{p.Iso.energy}
		For any (nonnegative) bounded measure $\mu$, one has
		\begin{equation} \label{e.Iso.Fenergy}
			\Iso_{\omega,\I} \Fenergy(X_N,\mu) = \Fenergy(X_N,\mu) -  \sum_{i \in \I} \h_\omega^{\sum_{j\ne i} \delta_{x_j} -  \mu}(x_i),
		\end{equation}
		where we define the Dirichlet electrostatic potential in $\omega$ generated by a signed measure $\nu$ with either bounded negative part or bounded positive part by
		\begin{equation} \label{e.homega.def}
			\h_\omega^\nu(x) := \int_{\omega} \g_\omega(x,y) \nu(dy),
		\end{equation}
		and $\frac{1}{\cd} \g_\omega$ is the Dirichlet Green's function of $-\Delta$ in $\omega$. In particular, since $\g_\omega$ is nonnegative, we have
		\begin{equation} \label{e.Iso.Fenergy.simple}
			\Iso_{\omega,\I} \Fenergy(X_N,\mu) \leq \Fenergy(X_N,\mu) +  \sum_{i \in \I} \h_\omega^{\mu}(x_i).
		\end{equation}
		Finally, one has
		\begin{equation}  \label{e.Iso.Henergy}
		\Iso_{\omega,\I} \Henergy^V_N(X_N) = \Henergy^V_N(X_N) - \sum_{i \in \I} \h_\omega^{\sum_{j \ne i} \delta_{x_j} - \cd^{-1}\Delta V}(x_i) \leq  \Henergy^V_N(X_N)  +  \sum_{i \in \I} \h_\omega^{\cd^{-1} \Delta V}(x_i).
		\end{equation}
	\end{proposition}
	\begin{proof}
		It is enough to consider $\I = \{1 \}$ and $x_1 \in \omega$. The energy $\Fenergy(X_N,\mu)$ can be written as a function of $(x_i)_{i \ne 1}$ plus
		$$
		\int_{\R^\d} \g(x_1 - y) \(\sum_{i=2}^N \delta_{x_i}- \mu_\infty \)(dy).
		$$
		It is clear then that the RHS of \eref{Iso.Fenergy} solves \eref{Iso.BVP} for $f(x_1)= \Fenergy(X_N,\mu)$. Similarly, one can check that the middle term of \eref{Iso.Henergy} solves \eref{Iso.BVP} for $f(x_1) = \Henergy^V_N(X_N)$.
	\end{proof}
	
	Next, we compute the adjoint of $\Iso_\omega$.
	\begin{proposition} \label{p.Iso.adjoint}
		Let $F : \omega \to \R$ and $G : \pa \omega \to \R$ be nonnegative and measurable. Then
		$$
		\int_{\omega} F(x) (\Iso_{\omega} G)(x) dx = \int_{\pa \omega} (\Iso_{\omega}^\ast F)(x) G(x)\Haus_{\pa \omega}(dx)
		$$
		where
		$$
		\Iso_{\omega}^\ast F(x) = \int_{\omega} F(y) \harm_\omega(y,x) dy.
		$$
	\end{proposition}
	\begin{proof}
		This is a direct application of Tonelli's theorem using the characterization of $\Iso_\omega$ in terms of harmonic measure.
	\end{proof}

	We are now ready for the central technical result of the paper.	For this, define the conditional $1$-point function $\rho_1(y | y_2,\ldots,y_k)$ as a properly rescaled Lebesgue density of the regular conditional probability $\P(x_1 \in dy | x_2 = y_2, \ldots, x_k = y_k)$, i.e.\ this solves
	\begin{equation}
		\rho_k(y,y_2,\ldots,y_k) = \rho_{k-1}(y_2,\ldots,y_k) \rho_1(y | y_2,\ldots,y_k) \quad y\text{-a.e.}
	\end{equation}
almost surely with $(y_2,\ldots,y_k)$ drawn from the $(k-1)$-particle marginal of the Coulomb gas.

	\begin{proposition} \label{p.1pt.Iso}
		Let $\tilde \rho_1(x) = \rho_1(x | y_2,\ldots,y_k)$, and recall the Dirichlet electrostatic potential $\h^\cdot_\omega$ from \eref{homega.def}.
		Then a.s.\ with $(y_2,\ldots,y_k)$ drawn from the $(k-1)$-particle marginal of the Coulomb gas, we have for $x \in \R^\d$ a.e.\ that
		\begin{equation} \label{e.1pt.Iso.old}
			\tilde \rho_1(x) \leq \exp\(\beta \h_\omega^{ \cd^{-1}\Delta V - \sum_{i=2}^k \delta_{y_k}}(x) \) \int_{\pa \omega} \tilde \rho_1(z) \harm_\omega(x,dz),
		\end{equation}
		and
		\begin{equation} \label{e.1pt.Iso.zeta}
			e^{\beta  \zeta(x)}\tilde \rho_1(x) \leq \exp\(\beta \h_\omega^{ \mu_\infty  - \sum_{i=2}^k \delta_{y_k}}(x) \)\int_{\pa \omega} e^{\beta  \zeta(z)} \tilde \rho_1(z) \harm_\omega(x,dz).
		\end{equation}
	\end{proposition}
	\begin{proof}
		First we prove \eref{1pt.Iso.zeta}. The $x \not \in \omega$ case is trivial, so we assume $x \in \omega$. Let $r >0$ be small enough that $B_{2r}(x) \subset \omega$, and let $M_r = \sup_{y \in B_r(x)} \h^{\mu_\infty - \sum_{i=2}^k \delta_{y_i} }_\omega(y)$. For any $(x_2,\ldots,x_N)$ we have by \eref{Iso.Fenergy} that
		\begin{align}
			\int_{B_r(x)} e^{-\beta \Fenergy(X_N, \mu_\infty) - \beta  \zeta(x_1)} dx_1 &\leq e^{\beta  M_r} \int_{B_r(x)} e^{-\beta (\Iso_{\omega,1} \Fenergy)(X_N,\mu_\infty) - \beta  \zeta(x_1)} dx_1 \\ \notag
			&\leq e^{\beta  M_r} \int_{B_r(x)} \(\Iso_{\omega,1} e^{-\beta \Fenergy(\cdot,x_2,\ldots,x_N,\mu_\infty)} \)(x_1) e^{-\beta  \zeta(x_1)} dx_1 \\ \notag
			&= e^{\beta M_r} \int_{\pa \omega} e^{-\beta \Fenergy(X_N,\mu_\infty)} \Iso_{\omega,1}^\ast \(e^{-\beta  \zeta}\1_{B_r(x)} \)(x_1) \Haus_{\pa \omega}(dx_1),
		\end{align}
		where we applied Jensen's inequality in the second to last line, using that $\harm_\omega(x_1,\cdot)$ is a probability measure for each $x_1 \in \omega$.
		
		Using \pref{Iso.adjoint} we compute
		$$
		\Iso_{\omega,1}^\ast \(e^{-\beta  \zeta}\1_{B_r(x)} \)(x_1) = \int_{B_r(x)} \harm_\omega(y,x_1) e^{-\beta \zeta(y)} dy =: J_r(x_1) {e^{-\beta  \zeta(x)} \harm_\omega(x,x_1)},
		$$
		defining $J_r(x_1)$ accordingly. We have proved
		\begin{align}
			\lefteqn{ \int_{B_r(x)} e^{-\beta \Fenergy(X_N, \mu_\infty) - \beta \zeta(x_1)} dx_1 } \quad & \\ \notag &\leq e^{\beta  M_r - \beta  \zeta(x)} \int_{\pa \omega} e^{-\beta \Fenergy(X_N,\mu_\infty) }  J_r(x_1) \harm_\omega(x,dx_1).
		\end{align}
		We can then integrate against $e^{-\beta  \sum_{i=2}^N \zeta(x_i)} dx_{k+1} \cdots dx_N$, set $(x_2,\ldots,x_k) = (y_2,\ldots,y_k)$, and divide by a normalizing factor to see (note the normalization $\int \tilde \rho_1 = N-k+1$):
		\begin{equation} \label{e.1pt.prelimit}
			\P\(x_1 \in B_r(x) | y_2,\ldots,y_k\) \leq \frac{e^{\beta  M_r - \beta  \zeta(x)}}{N - k + 1} \int_{\pa \omega}  J_r(z)  e^{\beta \zeta (z)}\harm_\omega(x,z) \tilde \rho_1(z) dz.
		\end{equation}
		
		We will now take $r \to 0$. Clearly $M_r$ decreases to $\h_\omega^{\mu_\infty - \sum_{i=2}^k \delta_{y_i}}(x)$ as $r \to 0$, and we have
		\begin{equation} \label{e.Jr.ratio}
		\frac{J_r(z)}{|B_r(z)|} =  \frac{1}{|B_r(x)|} \int_{B_r(x)} \frac{\harm_\omega(y,z)e^{-\beta  \zeta(y)}}{\harm_\omega(x,z)e^{-\beta  \zeta(x)}} dy.
		\end{equation}
		The integrand on the RHS converges uniformly to $1$ as $y \to x$ by continuity of $\zeta$ and of $\harm_\omega(\cdot,z)$. Indeed, $\zeta$ inherits the regularity of $V$ and $\harm_\omega(\cdot,z)$ satisfies a Harnack inequality, uniformly in $z \in \pa \omega$, on $B_{R/2}(x)$ for $R = \dist(x,\pa \omega)$. We thus see that $J_r(z)/|B_r(z)| \to 1$ as $r \to 0$, and also that the quantity in \eref{Jr.ratio} is uniformly bounded for $r \to 0$ and $z \in \pa \omega$. We may then conclude that
		\begin{align}
		\tilde \rho_1(x) &\leq \limsup_{r \to 0}\frac{ (N-k+1)\P\(x_1 \in B_r(x) | y_2,\ldots,y_k\)}{|B_r(x)|} \\ \notag &\leq {e^{\beta  \h_\omega^{\mu_\infty - \sum_{i=2}^k \delta_{y_i}}(x)- \beta  \zeta(x)}} \int_{\pa \omega} e^{\beta N \zeta(z)} \tilde \rho_1(z) \harm_\omega(x,dz),
		\end{align}
		finishing the proof of \eref{1pt.Iso.zeta}.
		
		The justification of \eref{1pt.Iso.old} is similar, except we instead use \eref{Iso.Henergy} in place of \eref{Iso.Fenergy}. The proof otherwise corresponds directly with $\Henergy^V_N(X_N)$ taking the role of $\Fenergy(X_N,\mu_\infty)$ and $0$ taking the role of $\zeta$.
	\end{proof}
	
	
	\begin{remark} \label{r.weightedpoly}
		It is interesting to compare our method to that of the previous best localization estimate of Ameur \cite{A21}. For Ameur, a key fact is that in $\d=2$ there are random weighted polynomials $P = P_{X_N}$ on $\mbb C \cong \R^2$ such that $\E P(z) = \rho_1(z)$, roughly speaking. These polynomials are of the form $q e^{-\beta N V_1}$ for a holomorphic polynomial $q$ of degree $N-1$, and they satisfy a maximum principle for each fixed configuration $X_N$. This naturally leads to the need to control the quantity $\E \| P \|_{L^\infty(\Sigma_1)}$, since the $L^\infty$ norm appears in each configuration-wise application of the maximum principle.
		
		\pref{1pt.Iso} can be viewed as a stochastic version of the well-known maximum principle for weighted polynomials (written in a mean value inequality form). That is, the maximum principle comes after the expectation, and we are lead to control quantities like $\| \E P \|_{L^\infty(\Sigma_1)}$ rather than $\E \|  P \|_{L^\infty(\Sigma_1)}$. This form will allow us to save a factor of $N$ in our $\rho_1$ estimate and extend  results to $\d \geq 3$. We also obtain more information due to the precise spatial dependencies detailed within \pref{1pt.Iso} as compared to a maximum principle.
	\end{remark}
	
	\begin{remark}
		The error $\omega \mapsto {\h^{\mu_\infty}_\omega}$ and its dependence on $\omega$ in \pref{Iso.energy} has interesting properties due to its connection with Brownian motion and its strong Markov property. For instance, one can show that the order of $\h_\omega^{\mu_\infty}$ is $O(r^2)$ whenever $\omega$ is a set with ``thickness" of order $r$, formalized as: Brownian motion, regardless of starting point, exits $\omega$ within time $r^2$ with probability at least some uniform $\delta > 0$. This raises the possibility of applying \pref{Iso.energy} to sets $\omega$ with large aspect ratios. One can also let $\omega$ be random and $\{x_i : i \not \in \I\}$ measurable.
	\end{remark}
	
	\section{Microscopic point process} \label{s.micro}
	In this section, we prove that weak limit points of (translates of) $\pp X_N = \sum_{i=1}^N \delta_{x_i}$ under the Coulomb gas exist as $N \to \infty$. To do so, it is sufficient to prove that the law of $\pp X_N(B_R(x))$ is tight uniformly in $N$ (by the compactness criterion \cite[Theorem 11.1.VI]{DVJ08}) for all $R > 0$. Such a result is known if $\inf_N \beta_N> 0$ by \cite[Theorem 1]{AS21} or \cite[Theorem 1]{T23}. In the high temperature regime, it is known from these results that if $R = R_N \geq \beta_N^{-1/2}$ then the law of $R^{-\d}\pp X_N(B_R(x))$ is uniformly tight in $N$, so the task is to extend this result down to smaller scales. This is done by \tref{mgf}, proved in \ssref{ptwise.upper} where we also conclude an upper bound on the $k$-point function that shows particle repulsion effects even in the high temperature regime.
	
	In \ssref{hightemp}, we study limits of $\pp X_N$ when $\beta_N\to 0$ as $N \to \infty$, proving \tref{poisson2}.
	
	\subsection{Density upper bounds} \label{ss.ptwise.upper}
	We will now seek to bound pointwise $\rho_1$ (and $\rho_k$). Our main focus will be bounds within $\Sigma$, where the bounds will be of the correct order, but the structure of $V_1$ will come into the argument solely through the quantity $$\mcl M_x = \| \max(\Delta V_1,0) \|_{L^\infty(B_{1}(N^{-1/\d}x))}.$$  It is also helpful to define the smaller quantity 
	\begin{equation} \label{e.tildeM.def}
	\tilde {\mcl M}_x =  \| \max(\Delta V_1,0) \|_{L^\infty(B_{\frac{1}{100}}( N^{-1/\d}x))}.
	\end{equation}
	
	The following result shows that the particle density is controlled on balls of radius $\max(\beta^{-1/2},1)$.
	\begin{theorem}[ \cite{T23}, Theorem 1] \label{t.t23}
		For any $R \geq 1$, integer $\lambda \geq 100$, and integer $Q$ satisfying
		\begin{equation} \label{e.K.cond.1}
			Q \geq \begin{cases}
				\frac{C \lambda^2 R^2 + C\beta^{-1} }{\log(\frac14 \lambda)} \quad &\text{if } \d=2,\\
				C  R^{\d} +C \beta^{-1} R^{\d-2} \quad &\text{if } \d \geq 3,
			\end{cases}
		\end{equation}
		we have
		\begin{equation}
			\P^{V}_{N,\beta}(\pp X_N(B_R(x)) \geq Q ) \leq \begin{cases}
				e^{-\frac12 \beta \log(\frac14 \lambda) Q^2 + C(1+\beta \lambda^2 R^2) Q} \quad &\text{if } \d=2,\\
				e^{-2^{-\d} \beta R^{-\d + 2} Q(Q-1)} \quad &\text{if } \d \geq 3.
			\end{cases}
		\end{equation}
		The constant $C$ depends only on $\mcl M_x$ and $\d$. The same result holds (with the same constant) if we add an arbitrary perturbation to $V$ which is superharmonic in each $x_i$; see \cite{T23} for details.
	\end{theorem}
	
	We also have the isotropic averaging result which follows directly from \pref{1pt.Iso}. It allows estimates for the $k$-point function to be proved by appeal to estimates at a larger scale $r$, with a bounded associated cost if $r = O( \beta^{-1/2})$.
	\begin{proposition} \label{p.kpt.comp}
		Let $y_1,y_2,\ldots,y_k \in \R^\d$ and $r > 0$. Define
		$$
		\mcl F(y_1,\ldots,y_k; r) = \frac12 \sum_{i \ne j} \max(0,\g(y_i - y_j) - \g(r/2)).
		$$
		One has
		\begin{align}  \label{e.kpt.comp}
			\lefteqn{ \rho_k(y_1,\ldots,y_k)} \quad & \\ \notag & \leq C^{k} r^{-\d k}  e^{-\beta \mcl F(y_1,\ldots,y_k;r) +  C\beta r^2 \sum_{i=1}^k \tilde{ \mcl M}_{y_i}}\int_{B_{r}(y_1) \times B_{r}(y_2) \times \cdots \times B_r(y_k)} \rho_k(z_1,\ldots,z_k) dz_1 \cdots dz_k
		\end{align}
		for a dimensional constant $C$. In particular, for $4s \leq r \leq N^{1/\d}/200$ one has
		\begin{equation} \label{e.binom.comp}
			\E \left[ \binom{\pp X_N(B_s(y))}{k} \right] \leq C^{k} e^{-\beta \binom{k}{2} (\g(2s) - \g(r/2)) + C \beta k r^2  {\mcl M}_y}\(\frac{s}{r}\)^{\d k}\E\left [\binom{\pp X_N(B_r(y))}{k} \right].
		\end{equation}
	\end{proposition}
	\begin{proof}
		By iteration, to prove \eref{kpt.comp} it will be enough to prove
		\begin{equation} \label{e.kpt.comp.iter}
		\tilde \rho_1(y_1) \leq C r^{-\d}e^{A} \int_{B_r(y_1)} \tilde \rho_1(z) dz
		\end{equation}
		for $A := C \beta r^2 \tilde{ \mcl M}_{y_1} - \beta \sum_{j=2}^k \max(0,\g(y_1 - y_j) - \g(r/2))$ and
		where $\tilde \rho_1$ is the one-point function of the Coulomb gas conditioned on $y_2,\ldots,y_k$ (as in \pref{1pt.Iso}).
		
		This estimate follows from \pref{1pt.Iso} with $\omega_s = B_s(y_1)$, $\frac{r}{2} \leq s \leq r$. Indeed, using the well-known explicit formula for $\g_{\omega_s}(y_1,\cdot)$, one can estimate
		$$
		\h^{\cd^{-1} \Delta V}_{\omega_s}(y_1) \leq C \| \max(\Delta V,0) \|_{L^\infty(B_s(y_1))} s^2 \leq C \tilde{ \mcl M}_{y_1} r^2,
		$$
		and
		$$
		\h^{\delta_{y_j}}_{\omega_s}(y_1) =  \g_{\omega_s}(y_1,y_j) \geq \max(\g(y_1 - y_j) - \g(r/2),0).
		$$
		It follows from \eref{1pt.Iso.old} that
		 $$\tilde \rho_1(y_1) \leq C s^{-\d+1} e^A\int_{\pa B_s(y_1)} \tilde \rho_1(z) \Haus_{\pa B_s(y_1)}(dz).$$
		 Integrating against $s^{\d-1} ds$ for $s \in [r/2,r]$ shows
		$$
		\tilde \rho_1(y) \leq C r^{-\d}e^{A} \int_{B_r(y_1) \setminus B_{r/2}(y_1)} \tilde \rho_1(z) dz
		$$
		which implies \eref{kpt.comp.iter}.
		
		Equation \eref{binom.comp} follows from integration of \eref{kpt.comp} (with $r$ replaced by $r - s$) over $(B_s(y))^k$. The expectations of the binomial terms in \eref{binom.comp} is what the $k$-fold integration of the $k$-point function computes, up to a combinatorial factor, and one estimates
		\begin{align*}
			&\int_{B_s(y)^k} \int_{B_{r-s}(y_1) \times \cdots \times B_{r-s}(y_k)} \rho_k(z_1,\ldots,z_k) dz_1 \cdots dz_k dy_1 \cdots dy_k  \\ &\quad \leq C^k s^{\d k} \int_{B_r(y)^k}\rho_k(z_1,\ldots,z_k) dz_1 \cdots dz_k.
		\end{align*}
	Furthermore, we clearly have for all $(y_1,\ldots,y_k) \in (B_s(y))^k$ that $\sum_{i=1}^k  \tilde{\mcl M}_{y_i} \leq k  {\mcl M}_y$, and
	$$
	\mcl F(y_1,\ldots,y_k;r) \geq \binom{k}{2} (\g(2s) - \g(r/2))
	$$
	since $|y_i - y_j| \leq 2s$.
	\end{proof}

	\begin{proof}[Proof of \tref{mgf}]
		We will bound the exponential moments of $\pp X_N(B_r(x))$.
		Let $R = \max(\beta^{-1/2},3)$. We allow implicit constants $C$ to depend continuously on $\gamma$ and $\mcl M_x$ below, but they are independent of $\beta$.
		
		We first work in the case that $r \geq R/3 \geq 1$. By \tref{t23} and some routine estimates, there exists $\lambda > 0$ (uniform in $r$) such that for any integer $Q$ with $\frac{Q}{r^\d} \geq C$, we have
		\begin{equation}
			\P(\pp X_N(B_r(x)) \geq Q) \leq e^{-9\lambda \beta r^{-d+2} Q^2} \leq e^{-\lambda \frac{Q^2}{r^\d}}.
		\end{equation}
		By summation by parts, we compute for any $\gamma > 0$ that
		\begin{align*}
			\E\left[ e^{\gamma \pp X_N(B_r(x))} \right] &=1 +  (1 - e^{-\gamma}) \sum_{Q=1}^\infty e^{\gamma Q} \P(\pp X_N(B_r(x)) \geq Q) \\ &\leq 1 + C r^\d e^{C\gamma r^\d} + \sum_{Q = \lceil Cr^\d \rceil}^\infty e^{\gamma Q - \lambda \frac{Q^2}{r^\d}}.
		\end{align*}
	 If we choose $C$ large enough based on $\frac{\gamma}{\lambda}$, one can bound the rightmost sum above by
		$$
		\sum_{Q = \lceil Cr^\d \rceil}^\infty e^{-\lambda Q} \leq C.
		$$
		If $\gamma \geq 1$ (and since $r \geq 1$), we conclude $\E\left[ e^{\gamma \pp X_N(B_r(x))} \right] \leq e^{C \gamma r^\d}$. If $\gamma < 1$, we use Jensen's inequality to see
		$$
		\E\left[ e^{\gamma \pp X_N(B_r(x))} \right] \leq 	\E\left[ e^{ \pp X_N(B_r(x))} \right]^{\gamma} \leq e^{C\gamma  r^{\d}}.
		$$
		
		We now handle the case $r < R/3$.
		By \eref{binom.comp} and the formula $e^{\gamma n} = \sum_{k=0}^\infty (e^{\gamma} - 1)^k \binom{n}{k}$ for any nonnegative integer $n$, we have
		\begin{align}
			\E\left[ e^{\gamma \pp X_N(B_r(x))} \right] &= \sum_{k=0}^\infty (e^{\gamma} - 1)^k \E\left[\binom{\pp X_N(B_r(x))}{k} \right] \\  \notag &\leq \sum_{k=0}^\infty \(C(e^{\gamma} - 1)e^{C \beta  R^2 \mcl M_x} \(\frac{r}{R}\)^\d\)^k  \E\left[\binom{\pp X_N(B_R(x))}{k} \right] \\ \notag
			&= \E \left[ e^{m_{\gamma,r,R} \pp X_N(B_R(x))} \right],
		\end{align}
		for 
		\begin{equation} \label{e.mgamma.def} 
			m_{\gamma,r,R} = \log\(1 + C(e^{\gamma }- 1)  e^{C \beta R^2 \mcl M_x} \(\frac{r}{R}\)^\d\) \leq C^{1+\beta}  \(\frac{r}{R}\)^\d.
		\end{equation}
	Using the $r \geq R/3$ case, we can bound
		$$
		\E\left[ e^{\gamma \pp X_N(B_r(x))} \right] \leq   \E \left[ e^{C^{1+\beta} \(\frac{r}{R}\)^{\d}\pp X_N(B_R(x))} \right] \leq e^{C^{1+\beta}r^\d},
		$$
		which gives the desired upper bound for $\gamma  \geq 1$. For $\gamma < 1$, we can use Jensen's inequality as earlier in the proof.
	\end{proof}
	
	An easy consequence of \tref{mgf} is a bound on the one point function.
	\begin{corollary} \label{c.rho1bd}
		One has $\rho_1(x) \leq C^{1+\beta}$ where $C$ depends only on $\mcl M_x$.
	\end{corollary}
\begin{proof}
	\tref{mgf} shows the existence of $C$, dependent only on $\mcl M_x$ such that
	$$
	\limsup_{r \to 0^+} r^{-\d}\E\left[  \pp X_N(B_r(x))) \right]  \leq	\limsup_{r \to 0^+} r^{-\d}\E\left[ \exp( \pp X_N(B_r(x))) - 1\right] \leq C^{1+\beta}.
	$$
\end{proof}
	
	\begin{remark}
		A variety of results from \cite{T23}, including bounds for $\rho_k$, particle clusters, and minimal interparticle distances, can be extended to high temperature regimes by using \tref{mgf} (typically with $r=1$) in place of \tref{t23} in various places within the proofs.
	\end{remark}
	
	\subsection{High temperature limit points} \label{ss.hightemp}
	With \tref{mgf} in hand, we can prove that limit points of the law of $\pp X_N$ exist in a certain sense even if $\beta_N\to 0$ as $N \to \infty$. These high temperature limit points inherit versions of \pref{1pt.Iso} obtained by formally setting $\beta = 0$. In $\d = 2$, we can exploit the structure of bounded subharmonic functions (i.e.\ constants) to give a rigid structure to the possible limit points.
	
	\begin{proposition} \label{p.hot.limits1}
		Let $\beta_N \to 0$ as $N \to \infty$, and let $x^\ast_N \in \R^\d$, $N \geq 1$, be a sequence with $N^{-1/\d} x^\ast_N$ bounded. Then the laws of $\tau_{x^\ast_N}\pp X_N$ under $\P^{V}_{N,\beta_N}$, $N \geq 1$, are tight under the weak topology on point processes. Limit points $\pp X$ are such that $\pp X(B_r(y))$ has a finite moment generating function with
		\begin{equation} \label{e.X.mgf}
			\E\left[ e^{\gamma \pp X(B_r(y))} \right] \leq e^{C \gamma r^\d}.
		\end{equation}
		uniformly over $r > 0$, $y \in \R^\d$, and $\gamma$ in a compact interval. The $k$-point correlation function $\rho_k$ of $\pp X$ is in $L^\infty(\R^\d)$ and is subharmonic in each variable, in the sense that
		\begin{equation} \label{e.X.subharm}
			x \mapsto \int_{A_1 \times \cdots \times A_k}  \rho_k(y_1,\ldots,y_{i-1}, y_i + x, y_{i+1},\ldots,y_k) dy_1 \cdots dy_k
		\end{equation}
		is a continuous subharmonic function for each $i=1,\ldots,k$ and any bounded open sets $A_1,\ldots,A_k$ with Lebesgue measure $0$ boundaries.
	\end{proposition}
	\begin{proof}
		We note that boundedness of $N^{-1/\d} x^\ast_N$ (and thus $\mcl M_{x^\ast_n,N}$) allows us to apply \tref{mgf} to show that
		$$
		\limsup_{N \to \infty} \E\left[ e^{\gamma\tau_{x^\ast_N}\pp X_N(B_r(x)) } \right] \leq e^{C \gamma r^\d}
		$$ 
		for a constant $C$ uniform in $x \in \R^\d$ and $r > 0$ and bounded $\gamma$. Tightness of the law of $\tau_{x^\ast_N}\pp X_N$ then follows from the tightness criterion from \cite[Theorem 11.1.VI]{DVJ08}. In particular, along appropriate subsequences the law of $\tau_{x^\ast_N} \pp X_N(B_r(x))$ converges, and so \eref{X.mgf} holds for any limit point. The $L^\infty$ bound for $\rho_k$ for limit points follows from this argument as well in a way similar to \cref{rho1bd}.
		
		We now prove the subharmonicity property \eref{X.subharm}, where it will be sufficient to consider $i=1$. Let $q$ be the function in \eref{X.subharm}, and let $q_N$ be defined by
		$$
		q_N(x) =  \int_{A_1 \times \cdots \times A_k}\tau_{x^\ast_N}^{\otimes k} \rho_{k,N}(x + y_1,y_2,\ldots,y_k) dy_1 \cdots dy_k.
		$$
		Here $\tau_{x^\ast_N}^{\otimes k} \rho_{k,N}$ is the $k$-point correlation function of $\tau_{x^\ast_N}\pp X_N$. By \eref{1pt.Iso.old} (applied to the $1$-point function for the translated gas conditioned on $x_j = y_i$, $j =2,\ldots,k$) and switching the order of integration, we have for any $r > 0$ that
		$$
		q_N(x) \leq \sup_{y \in A_1} e^{\beta_N \h_{B_r(x + y - x_N^\ast)}^{\cd^{-1}\Delta V_N}(x+y-x_N^\ast)} \frac{1}{\Haus(\pa B_r(x))}\int_{\pa B_r(x)} q_N(z) \Haus_{\pa B_r(x)}(dz).
		$$
		Since $N^{-1/\d}x_N^\ast$ and $A_1$ is bounded, the quantity $\beta_N\Delta V_N(z) = \beta_N (\Delta V_1)(N^{-1/\d}z)$ converges to $0$ as $N \to \infty$ for $z$ in the relevant sets above uniformly over $N^{-1/\d}x$ and $r$ in a compact set. Furthermore, $q_N \to q$ pointwise as $N \to \infty$ (subsequentially) since $\pp X$ has bounded correlation functions and the $A_i$ have measure $0$ boundaries, and $q_N$ is  bounded on compact sets uniformly in $N$. By weak convergence, we conclude that
		$$
		q(x) \leq \frac{1}{\Haus(\pa B_r(x))}\int_{\pa B_r(x)} q(z) \Haus_{\pa B_r(x)}(dz) \quad \forall r > 0,
		$$
		which means that $q$ is subharmonic. That $q$ is continuous follows from our $L^\infty$ bound on $\rho_k$ and the fact that the symmetric difference of $A_1 + x$ and $A_1$ has vanishing measure as $|x| \to 0$.
	\end{proof}
	
	\begin{corollary} \label{c.constant}
		Let $\d = 2$ and let $\pp X$ be a limit point as in \pref{hot.limits1}. Then the $k$-point correlation functions of $\pp X$ are constant.
	\end{corollary}
	\begin{proof}
		The function defined in \eref{X.subharm} for $i=1$ is subharmonic and uniformly bounded above. It is a classical fact that all such functions are constant in two dimensions. Since the $A_i$, $1 \leq i \leq k$, can be any bounded open sets, we must have $\rho_k(y_1,\ldots,y_k) = \rho^{(1)}_k(y_2,\ldots,y_k)$ for some $L^\infty$ function $\rho_k^{(1)}$. The subharmonicity of \eref{X.subharm} for $i=2$ then implies that $x \mapsto \rho^{(1)}_k(x + y_2,\ldots,y_k)$ is subharmonic, in a similar integral sense, and we can iterate the argument to eliminate the dependence on $y_2$, and so on.
	\end{proof}

We are ready to prove our main theorem on high temperature limit points in $\d = 2$. Indeed, all that is left is to verify that point processes with constant correlation functions are mixed Poisson processes.
	
	\begin{proof}[Proof of \tref{poisson2}]
		By \cref{constant} and \eref{X.mgf}, the $k$-point correlation functions of $\pp X$ are constant and the moment generating function of $M_R = (\pi R^2)^{-1} \pp X(B_R(0))$ is locally bounded uniformly in $R \in (1,\infty)$. Let $\mu$ be any subsequential limit of the law of $M_R$, say along a sequence $R_n \to \infty$.
		
		For any $k \geq 1$, the uniform exponential tails of $M_R$ imply
		$$
		\E \left[ M_R^k \right] = \int_0^\infty m^k \mu(dm) + o(1)
		$$
		as $R = R_n \to \infty$. Letting $P_k(x) = x(x-1)\cdots(x-k+1)$, we see
		\begin{equation} \label{e.muk1}
		\E \left[ P_k(\pp X(B_R(0)))\right] = \int_0^\infty P_k\(\pi R^2 m\) \mu(dm) + o(R^{2k}) =  (\pi R^2)^k \int_0^\infty m^k \mu(dm)+ o(R^{2k}),
		\end{equation}
		and we also have
		\begin{equation} \label{e.muk2}
		\E \left[ P_k(\pp X(B_R(0)))\right]  = \rho_k(B_R(0)^k)= \rho_k \cdot (\pi R^2)^k,
		\end{equation}
		where $\rho_k$ also denotes the (constant) Lebesgue density of $\rho_k$. Comparing \eref{muk1} and \eref{muk2}, we see that $\rho_k$ is the $k$th moment of $\mu$, which matches that of a mixed Poisson process with constant intensity drawn independently from $\mu$. In particular, $M_R$ converges in law to $\mu$ as $R \to \infty$, and in the case that $\mu$ is supported on a singleton $\{m_0\}$, the process $\pp X$ is Poissonian of constant intensity $m_0$.
	\end{proof}

	\section{Confinement of the Coulomb gas} \label{s.confinement}
	In this section, we prove quantitative results showing that the particles in the Coulomb gas are typically confined to a small neighborhood of the droplet $\Sigma$. In \ssref{confine.main}, we prove our main result on $\rho_1$ assuming that $\rho_1$ has strong decay at $\infty$. In \ssref{confine.vacuum}, we apply the main result and a union bound to bound the distance of furthest particle from the vacuum. In \ssref{confine.extreme2} and \ssref{confine.extreme3}, we complete the results by examining the behavior of $\rho_1$ at $\infty$. In particular, the $\d \geq 3$ results use an interesting technique of transporting a particle from $\infty$ by ``squeezing" it into the space between other particles, and will require the particle repulsion bounds we proved in \pref{kpt.comp} to bound the transport cost. 
	
	
	\subsection{Proof of the main result} \label{ss.confine.main}
 The following proposition's proof will appear in \ssref{confine.extreme2} and \ssref{confine.extreme3}.
	
	\begin{proposition} \label{p.farfield.simple}
		If $\d = 2$, we have
		$$
		\lim_{|x| \to \infty} e^{\beta \zeta(x)}\rho_1(x) = 0.
		$$
		(The convergence is not necessarily uniform in $N$ or $\beta$.) If $\d \geq 3$, we have $V$ dependent constants $C < \infty$ and $c > 0$ such that
		$$
		\limsup_{|x| \to \infty} e^{\beta \zeta(x)}\rho_1(x) \leq C^{1+\beta} \max(\log \beta^{-1},1) e^{-c \beta N^{2/\d}}
		$$
		for all $N \geq C$.
	\end{proposition}

	We are now ready to prove \tref{rho1bd}.
	\begin{proof}[Proof of Theorem \tref{rho1bd}]
		Let $x \not \in \Sigma$, and let $\omega_R = B_R(0) \setminus \Sigma$ with $R$ large enough that $x \in \omega_R$. Note $\h^{\mu_\infty}_{\omega_R}(x) = 0$ since $\mu_\infty = 0$ in $\omega_R$. Then by \pref{1pt.Iso} we may estimate
		\begin{equation} \label{e.mainconfine.step.first}
			e^{\beta \zeta(x)} \rho_1(x)\leq \int_{\pa B_R(0)} e^{\beta \zeta(z)} \rho_1(z) \harm_\omega(x,dz) + \int_{\pa \Sigma} \rho_1(z) \harm_\omega(x,dz).
		\end{equation}
		Here we used that $\zeta(z) = 0$ q.e.\ on $\pa \Sigma$ (which is stronger than $\Haus_{\pa \Sigma}$-a.e.). By \cref{rho1bd}, we have $\rho_1(z) \leq C^{1+\beta}$ on $\pa \Sigma$, and so taking $R \to \infty$ shows
		\begin{equation} \label{e.t1.mainconfine.step.d2}
			e^{\beta \zeta(x)} \rho_1(x)\leq C^{1+\beta} + \limsup_{|z| \to \infty} e^{\beta \zeta(z)} \rho_1(z) \harm_\omega(x,\pa B_{|z|}(0)).
		\end{equation}
		Applying \pref{farfield.simple} finishes the proof.
	\end{proof}
	
	\subsection{The distance to vacuum} \label{ss.confine.vacuum}
	We now prove \cref{vacuum} as a consequence of \tref{rho1bd}. Recall from \eref{zeta.assume} that there exists some $\alpha > 0$ such that
	\begin{equation} \label{e.close.zeta.growth1}
	\zeta_1(x) \geq \alpha  \min(\dist(x,\Sigma_1)^2, 1) \quad \forall x \in \R^\d,
	\end{equation}
	which we can rescale to see
	\begin{equation} \label{e.close.zeta.growth}
	\zeta(x) \geq \alpha \min(\dist(x,\Sigma)^2, N^{2/\d}) \quad \forall x \in \R^\d.
	\end{equation}
	
	\begin{proof}[Proof of \cref{vacuum}]
		Note that
		\begin{equation}
			\P\(\max_{i \in \{1,\ldots,N\}} \zeta(x_i) \geq \gamma\) \leq N 	\P\(\zeta(x_1) \geq \gamma\) = \int_{\{\zeta \geq \gamma\}} \rho_1(x)dx.
		\end{equation}
		Equation \eref{vacuum.hard} follows by plugging in \eref{rho1bd} or \eref{rho1bd3}. In $\d \geq 3$ we assume that $\beta N^{2/\d} \geq c_1^{-1}\log \log N$ for a small enough $c_1 > 0$, so \eref{rho1bd3} is controlled by $C^{1+\beta} e^{-\beta \zeta(x)}$. 
		
		Turning to \eref{vacuum.soft}, we may bound for any $\gamma \geq 0$:
		$$
			\int_{\{\zeta \geq \gamma\}} e^{-\beta \zeta(x)} dx  \leq e^{-\beta \gamma + 2N^{-2/\d} \gamma} \int_{\R^\d} e^{-2N^{-2/\d} \zeta(x)} dx.
		$$
		Since we are assuming here that \eref{V.assume} holds with $\theta_\ast = 4$, it is a simple calculation to see
		$$
		 \int_{\R^\d} e^{-2 N^{-2/\d} \zeta(x)} dx = N\int_{\R^\d} e^{-2 \zeta_1(x)} dx \leq CN.
		$$
		We conclude that $\P(\max_{i} \zeta(x_i) \geq \gamma) \leq C Ne^{-\beta \gamma + 2N^{-2/\d} \gamma}$. The growth estimate  \eref{close.zeta.growth} implies
		$$
		 \left\{ \max_{i \in \{1,\ldots,N\}} \dist(x_i, \Sigma) \geq \sqrt{c^{-1}\gamma} \right\} \subset \left \{\max_{i \in \{1,\ldots,N\}} \zeta(x_i) \geq \gamma \right\} \quad \forall  \gamma \in[0, cN^{2/\d}].
		$$
		for a small enough $c > 0$. The corollary follows from the change of variables $\gamma' \sqrt{\beta^{-1} \log N} = \sqrt{c^{-1} \gamma}$.
	\end{proof}
	
	\subsection{The thermal equilibrium measure} \label{ss.confine.thermal}
	It will be convenient, in the proof of \pref{farfield.simple} and related results, to work with the {\it thermal equilibrium measure} $\mu_{\theta,N}$ for $\theta > 0$. Recalling the electric potential generated by a measure \eref{epot.def}, this measure with total mass $N$ is defined through the relation
	\begin{equation} \label{e.mutheta.def}
		\h^{\mu_{\theta,N}} + V_N + \frac{N^{2/\d}}{\theta} \log \mu_{\theta,N} = c_{\theta,N}
	\end{equation}
	for some constant $c_{\theta,N}$, from which one can check that
	$\mu_{\theta,N}(A) = N \mu_{\theta,1}(N^{-1/\d} A)$ for measurable $A \subset \R^\d$ and fixed $\theta > 0$. As usual, we will let $\mu_{\theta} = \mu_{\theta,N}$ and use $\mu_\theta$ for both the measure and its Lebesgue density, e.g. $\mu_{\theta,N}(x) = \mu_{\theta,1}(N^{-1/\d}x)$. Unless otherwise stated, we set $\theta = \beta N^{2/\d}$ implicitly for now on. The choice to include $\theta$, rather than $\beta$, in the notation is so that $\mu_{\theta,N}$ approximates $\mu_{\infty,N}$ on the macroscopic length scale $N^{1/\d}$ as $N \to \infty$ so long as $\theta$ grows to $\infty$.
	
	The following proposition establishes the electric splitting formula with respect to $\mu_\theta$.
	\begin{proposition} \label{p.split.thermal}
		One has
		\begin{equation} \label{e.split.thermal}
			\exp\(-\beta \Henergy^{V}_N(X_N)\) = \exp\(-\beta \Eenergy(\mu_\theta,V) -\beta \Fenergy(X_N,\mu_\theta) \) \prod_{i=1}^N \mu_\theta(x_i).
		\end{equation}
	\end{proposition}
	\begin{proof}
		The proof is analogous to that of \pref{split}. Indeed, \eref{split.1} with $\mu_\theta$ in place of $\mu_\infty$ holds, and using the equation
		$$
		\h^{\mu_{\theta}} + V + \frac{1}{\beta} \log \mu_{\theta} = c_{\theta,N}
		$$
		finishes the proof.
	\end{proof}
	
	The following proposition, based on \cite[Theorem 1]{AS19} with some routine calculations, establishes a few useful properties of the (thermal) equilibrium measure. By convention, we allow $\mu_\theta$ with $\theta = \infty$ to be the equilibrium measure $\mu_\infty$ as in \eref{muthetaN.def}.
	\begin{proposition} \label{p.mu.basic}
		Letting $\theta_\ast$ be as in \eref{theta0.def}, the following holds for uniformly for $\theta \geq \theta_\ast > 2$. First one has for $\theta < \infty$ that
			\begin{equation} \label{e.ch.diff}
			\| \h^{\mu_{\theta}}- c_{\theta} - \h^{\mu_{\infty}} + c_{\infty} \|_{L^\infty(\R^\d)} \leq \frac{CN^{2/\d}}{\theta}.
		\end{equation}
		We also have
		\begin{equation} \label{e.convert}
			C^{-1} \leq \frac{\mu_{\theta}(x)}{e^{-N^{-2/\d} \theta \zeta(x)}} \leq C.
		\end{equation}
		For $\theta_\ast \leq \theta \leq \infty$, we have
		\begin{equation} \label{e.mu.upper}
			\| \mu_{\theta} \|_{L^\infty(\R^\d)} \leq C.
		\end{equation}
		If $\d \geq 3$
		\begin{equation} \label{e.h.bd3}
			0 \leq \h^{\mu_\theta} \leq CN^{2/\d},
		\end{equation}
		and if $\d = 2$
		\begin{equation} \label{e.h.bd2}
			\h^{\mu_\theta} \leq -\frac12 N \log N + CN \quad \text{and} \quad \exists C_\delta \quad \sup_{|x| \geq \delta} | \h^{\mu_{\theta,1}}(x) + \log |x| | \leq C_\delta \quad \forall \delta > 0.
		\end{equation}
		Finally, one has 
		\begin{equation} \label{e.self.theta}
		\int \h^{\mu_\theta} \mu_\theta = -\1_{\d=2}\frac12 N^2 \log N + O(N^{1+2/\d}).
		\end{equation}
	\end{proposition}
	\begin{proof}
		We begin by noting the scaling identity (valid for $\theta = \infty$)
		\begin{equation} \label{e.pot.scaling}
		\h^{\mu_{\theta,N}}(x) = N^{\frac{2}{\d}} \h^{\mu_{\theta,1}}(N^{-1/\d}x) - \1_{\d=2} \frac12 N \log N.
		\end{equation}
		from which it follows, using \eref{mutheta.def}, that
		$$
		c_{\theta,N} = N^{2/\d} \c_{\theta,1} - \1_{\d=2} \frac12 N \log N.
		$$
		Included in \cite[Theorem 1]{AS19} is the result (the $\beta$ in \cite{AS19} is equivalent to our $\theta$)
		$$
			\| \h^{\mu_{\theta,1}}- c_{\theta,1} - \h^{\mu_{\infty,1}} + c_{\infty,1} \|_{L^\infty(\R^\d)} \leq \frac{C}{\theta}.
		$$
		Equation \eref{ch.diff} follows.
		
		 Substituting into \eref{ch.diff} with \eref{mutheta.def} and \eref{zeta.prop} shows
		 $$
		 \| \frac{N^{2/\d}}{\theta} \log \mu_{\theta,N} + \zeta_N \|_{L^\infty(\R^\d)} \leq \frac{CN^{2/\d}}{\theta},
		 $$
		 which becomes \eref{convert} after exponentiation.
		
		The bound \eref{mu.upper} for $\theta < \infty$ follows from \eref{convert} and $\zeta \geq 0$ q.e.. At $\theta = \infty$, it follows from $\mu_{\theta,N} = \cd^{-1} \Delta V_N$ in the droplet and $\mu_{\theta,N} = 0$ outside. The upper bound for the positive part of $\h^{\mu_{\theta,N}}$ in \eref{h.bd3} and \eref{h.bd2} for $N=1$ follows from Young's inequality: 
		$$\h^{\mu_{\theta,1}}(x) \leq \max(\g,0) \ast \mu_{\theta,1}(x) \leq \| \max(\g,0) \|_{L^1(\R^\d)} \| \mu_{\theta,1} \|_{L^\infty(\R^\d)} \leq C.
		$$
		For $N > 1$, the result follows from the scaling identity \eref{pot.scaling}.
		Since $\g \geq 0$ in $\d \geq 3$, the lower bound for $\h^{\mu_{\theta,N}}$ is trivial.
		
		Considering now $\d = 2$ and the second bound in \eref{h.bd2}, we can compute for $A_x = \{y \ : \ |y| \leq \frac12 |x|\}$ that
		\begin{equation} \label{e.hbd.2.farstep}
			\left |\h^{\mu_{\theta,1}}(x) - \mu_{\theta,1}(A_x) \g(x) - \int_{A_x^c} \g(x-y) \mu_{\theta,1}(dy) \right | = \left| \int_{A_x} \( \g(x-y) - \g(x) \) \mu_{\theta,1}(dy) \right| \leq \log 2.
		\end{equation}
	The last bound follows from $|\log |x-y| - \log |x| | \leq \log 2$ on $A_x$.
		One may compute as $|x| \to \infty$
		$$
		(1 - \mu_{\theta,1}(A_x))|\g(x)| = \int_{A_x^c}| \g(x) |\mu_{\theta,1}(dy) \leq \int_{A_x^c} \log|y| \mu_{\theta,1}(dy) + \log 2,
		$$
		and (since $\g(z) \leq 0$ for $|z| \geq 1$)
		$$
		\int_{\{|x-y| \leq 1\}} \g(x-y) \mu_{\theta,1}(dy) \geq \int_{A_x^c} \g(x-y) \mu_{\theta,1}(dy) \geq -\int_{A_x^c} \log|y| \mu_{\theta,1}(dy) - C.
		$$
		By \eref{V.assume}, \eref{convert}, and the relation between $\zeta_1$ and $V_1$ near $\infty$ from \eref{muinfty.def}, we see from the above that
		$$
		\lim_{|x| \to \infty} \int_{A_x^c} \g(x-y) \mu_{\theta,1}(dy)  = 0.
		$$
		By combining the above with \eref{hbd.2.farstep}, we find $\h^{\mu_{\theta,1}}(x) = \g(x) + O(1)$ as $|x| \to \infty$.
		
		Finally, it is easy to see using \eref{h.bd2}, \eref{convert}, and \eref{V.assume} that
		$$
		\left | \int \h^{\mu_{\theta,1}} \mu_{\theta,1}\right | \leq C,
		$$
		uniformly in $\theta > \theta_\ast$, and \eref{self.theta} follows from scaling.
	\end{proof}
	
	\subsection{Proof of \pref{farfield.simple} in dimension two} \label{ss.confine.extreme2}
	Our proof of \pref{farfield.simple} will be based on the following lemma. We emphasize that we are conditioning on all particles except $x_1$ below. The general structure of the proof, namely the use of Jensen's inequality, seems to be well-known, used for instance in \cite{A21} for a similar purpose or in \cite{CHM18}, but we could not find a version that uses the thermal equilibrium measure or which gives exactly our conclusion.
	\begin{lemma} \label{l.farfield}
		Let $X_{N,1} = (x_2,\ldots,x_N)$ and $\delta > 0$. There is a constant $C = C_{V,\delta}$ such that for any measurable set $U$ with $\dist(U,0) \geq \delta N^{1/2}$, we have in $\d = 2$ that
		\begin{align} \label{e.farfield.reduction2}
			\P\(\{x_1 \in U\} \cap \{|x_1| \geq \delta \max(|x_2|,\ldots,|x_N|)\} \big | X_{N,1}\) \leq N^{-1} e^{C \beta N}  \int_{U} e^{-\beta \log |y| - \beta \zeta(y)} dy
		\end{align}
		almost surely. In $\d \geq 3$, we have a.s.\ that
		\begin{equation}\label{e.farfield.reduction3}
		\P\(x_1 \in U  \big | X_{N,1}\)  \leq N^{-1} e^{C\beta N^{2/\d}} \int_U e^{-\beta \zeta(y)}dy
		\end{equation}	
	for any measurable set $U$.
	\end{lemma}
	\begin{proof}
		Let $y \in \R^\d$ and $X_{N,1} = (x_2,\ldots,x_N)$ be fixed. By expanding $\Fenergy(\cdot,\mu_\theta)$, one has
		\begin{align}
			\frac{1}{N} \lefteqn{ \int_{\R^\d} \Fenergy((z,X_{N,1}),\mu_\theta) \mu_\theta(dz)} \quad &\\ \notag &= \Fenergy(X_{N,1},\mu_\theta) + \frac{1}{N}\iint_{\R^\d \times \R^\d} \g(z-z') \( \sum_{i=2}^N \delta_{x_i}(dz') -  \mu_\theta(dz')\) \mu_\theta(dz) \\ \notag
			&= \Fenergy((y,X_{N,1}), \mu_\theta) + \frac{1}{N}\iint_{\R^\d \times \R^\d}  \g(z-z') \( \sum_{i=2}^N \delta_{x_i}(dz') -  \mu_\theta(dz')\) \mu_\theta(dz)\\ \notag &\quad \quad \quad \quad \quad \quad \quad \quad \ \  - \int_{\R^\d} \g(y-z') \( \sum_{i=2}^N \delta_{x_i}(dz') -  \mu_\theta(dz')\) \\ \notag
			&=: \Fenergy((y,X_{N,1}), \mu_\theta) + N^{-1}\Fluct' \left[ \h^{\mu_\theta}\right] - \g \ast \fluct'(y),
		\end{align}
		where we defined $\fluct' =  \sum_{i=2}^N \delta_{x_i}(dz') -  \mu_\theta(dz')$ and $\Fluct' [ \cdot] = \int [\cdot] \fluct'$.
		Applying Jensen's inequality, with the convex function $\gamma \mapsto e^{-\beta \gamma}$, shows
		\begin{equation} \label{e.farfieldbasis}
			e^{-\beta \Fenergy((y,X_{N,1}),\mu_\theta)} \leq e^{\frac{\beta}{N}  \Fluct' \left[ \h^{\mu_\theta} \right] -  \beta  \g \ast \fluct'(y)} \frac{1}{N}\int_{\R^\d} e^{-\beta \Fenergy((z,X_{N,1}),\mu_\theta)} \mu_\theta(dz).
		\end{equation}
		Upon multiplying the above by $\mu_\theta(dy)$ and dividing by the integral in the RHS (when nonzero), the LHS becomes the density of the conditional probability $\P(x_1 \in \cdot | X_{N,1})$ evaluated at $y$; see \pref{split.thermal}. We find
		\begin{align} \label{e.farfield.cond}
			\P(x_1 \in dy | X_{N,1}) &\leq N^{-1}e^{\frac{\beta}{N}  \Fluct' \left[ \h^{\mu_\theta} \right] - \beta  \g \ast \fluct'(y)} \mu_\theta(dy) \\ \notag &\leq CN^{-1}e^{\frac{\beta}{N}  \Fluct' \left[ \h^{\mu_\theta} \right] - \beta  \g \ast \fluct'(y)} e^{-\beta \zeta(y)}dy,
		\end{align}
		where we changed $\mu_\theta$ to $e^{-\beta \zeta}$ using \eref{convert}. From \eref{h.bd2} or \eref{h.bd3} and \eref{self.theta}, we can bound
		$$
			\Fluct'[\h^{\mu_\theta}] \leq (N-1)  \sup_{x \in \R^\d} \h^{\mu_\theta}(x) - \int_{\R^\d} \h^{\mu_\theta} \mu_\theta \leq CN^{1+2/\d}.
		$$

		If $\d=2$ and $|y| \geq \delta \max(|x_2|,\ldots,|x_N|)$, then since $|y - x_i| \leq (1 + \delta^{-1})|y|$ and \eref{h.bd2} holds, we have
		\begin{align*}
			-\g \ast \fluct'(y) &\leq (N-1) \log(1 + \delta^{-1}) + (N-1) \log |y| + \h^{\mu_\theta}(y) \leq - \log |y| + C N
		\end{align*}
	as $N^{-1/2}|y| \to \infty$. If $\d \geq 3$, we simply bound $-\g \ast \fluct'(y) \leq \g \ast \mu_\theta(y) \leq CN^{2/\d}$.
	
		For $\d = 2$, inserting these bounds into \eref{farfield.cond} yields
		$$
		\P(x_1 \in dy | X_{N,1}) \leq CN^{-1}e^{C \beta N} e^{-\beta \zeta(y) - \beta \log |y|},
		$$
		and integrating this over the relevant set yields \eref{farfield.reduction2}. Similarly, we can prove \eref{farfield.reduction3} by combining the above.
	\end{proof}

\begin{remark} \label{r.farfield.fail}
	The bound \eref{farfield.reduction2} is already sufficient to prove \pref{farfield.simple} if $\d = 2$. The $\d \geq 3$ bound \eref{farfield.reduction3} does not give sufficient bounds for the $1$-point function function at $\infty$, but some more care allows us to prove
	\begin{equation} \label{e.farfield.reduction3.sharp}
		\limsup_{|y|\to \infty} e^{\beta \zeta(y)}\rho_1(y) \leq  C\E\left[ e^{\frac{\beta}{N} \Fluct[\h^{\mu_\theta}] }\right].
	\end{equation}
	for $\Fluct[\cdot] = \int [\cdot] \(\sum_{i=1}^N \delta_{x_i} - \mu_\theta\)$. This would allow us to prove a result sufficient for \tref{rho1bd} if we could show that the linear statistic $\Fluct[\phi]$ is typically not larger than order $\beta^{-1} N^{1-2/\d}$ with $\phi =  N^{-2/\d} \h^{\mu_\theta}$. Unfortunately such a bound is only known when $\phi$ is supported within the droplet, whereas the support of $\h^{\mu_\theta}$ extends past the droplet boundary; see \cite[Corollary 2.2]{S22}. We instead use a different approach for the $\d \geq 3$ result in \ssref{confine.extreme3}.
\end{remark}
	
	\begin{proof}[Proof of \pref{farfield.simple} in dimension two]
		We will bound the probability that $x_1 \in B_r(y)$ and then send $r \to 0$. We split into two cases:
		\begin{align}  \label{e.farfield2.cases}
			\P\(x_1 \in B_r(y)\) &\leq \P\(\{x_1 \in B_r(y)\} \cap  \{|x_1| \geq \frac12 \max(|x_2|,\ldots,|x_N|)\} \) \\ \notag &\quad + \P\(\{x_1 \in B_r(y)\} \cap  \{|x_1| < \frac12 \max(|x_2|,\ldots,|x_N|)\} \).
		\end{align}
		For the first term in the RHS above, by \lref{farfield} we have
		$$
		\limsup_{r \to 0}  \frac{1}{\pi r^2} \P\(\{x_1 \in B_r(y)\} \cap  \{|x_1| \geq \frac12 \max(|x_2|,\ldots,|x_N|)\} \) \leq N^{-1}e^{C\beta N} e^{-\beta \log |y| - \beta \zeta(y)}.
		$$
		For the second term in the RHS of \eref{farfield2.cases}, we bound (using particle exchangeability and \lref{farfield})
		\begin{align*}
			\lefteqn{\P\(\{x_1 \in B_r(y)\} \cap  \{|x_1| < \frac12 \max(|x_2|,\ldots,|x_N|)\} \) }\quad &\\ &\leq \P\(\{x_1 \in B_r(y)\}\) \\ &\quad \times N\P\(\{|x_2| \geq 2|y| -2r\} \cap \{|x_2| \geq \max(|x_1|,|x_3|,\ldots,|x_N|)\} | \{x_1 \in B_r(y)\}\) \\
			&\leq  e^{C \beta N}\P\(x_1 \in B_r(y)\) \int_{\{|z| \geq 2|y| - 2r\}} e^{-\beta \log |z| - \beta \zeta(z)} dz
		\end{align*}
		for all $r$ small enough. Letting $Q_y = e^{\beta \zeta(y)}\limsup_{r \to 0} \pi^{-1} r^{-2}\P(x_1 \in B_r(y))$, we have shown
		$$
		Q_y \leq N^{-1} e^{C \beta N} e^{-\beta \log |y|} + e^{C \beta N} Q_y \int_{\{ |z| \geq 2|y| \}} e^{-\beta \log |z| - \beta \zeta(z)} dz.
		$$
		Since the integral on the RHS converges to $0$ as $|y| \to \infty$, we have
		$Q_y \leq 2N^{-1} e^{C \beta N - \beta \log |y|}$ for all $|y|$ large enough (dependent on $N$). This establishes \pref{farfield.simple}.
	\end{proof}

	\subsection{Proof of \pref{farfield.simple} in dimension three and higher} \label{ss.confine.extreme3}
	Our goal in this section is to prove an upper bound for $\rho_1(x)$ in the limit $|x| \to \infty$ in $\d \geq 3$. Our method involves a new technique involving ``squeezing" a particle into the interstitial space between the other particles. We assume $\d \geq 3$ implicitly throughout, though many techniques should carry over to the $\d = 2$ case.
	
To be precise, let $\nu_{x,\eta}$ be the uniform probability measure on $B_\eta(x) \setminus B_{\eta/2}(x)$. Define $\nu = \nu_{X_{N,1}} = \frac{1}{N-1} \sum_{i=2}^N \nu_{x_i,\tilde \eta_i}$ where $$\tilde \eta_i = \frac14\min\(1,\min_{j \ne i,1} |x_i - x_j|\)$$
is one fourth the truncated distance between $x_i$ and the nearest distinct particle within $X_{N,1} = (x_2,\ldots,x_N)$. We also define $ \eta_i = \frac14\min\(1,\min_{j \ne i} |x_i - x_j|\)$.

The next proposition, roughly speaking, shows that ``replacing" $x_1$ by a unit charge shaped like $\nu$ changes the energy of the configuration by a small multiplicative factor plus some error. This operation will be used to compare probabilities of configurations in which $x_1$ is very far from $X_{N,1}$ to configurations in which $x_1$ is in interstitial spaces within $X_{N,1}$. Such a comparison allows us to find an upper bound for $\rho_1(x)$ for large $|x|$.  As to why the more obvious technique of ``replacing" a particle by a $\mu_\theta$-shaped charge is (currently) infeasible, see \rref{farfield.fail}.

\begin{proposition} \label{p.alpha.F.bd}
	One has
	\begin{align} \label{e.alpha.F.bd}
		\lefteqn{\int_{\R^\d}\(\Fenergy((y,X_{N,1}),\mu_\infty) + \zeta(y) + \sum_{i=2}^N \zeta(x_i) \) \nu(dy) } \quad & \\ \notag &\leq \(1 + \frac{1}{N-1}\)\(\Fenergy(X_{N,1},\mu_\infty) + \sum_{i=2}^N \zeta(x_i)\) - \frac{1}{2(N-1)} \int\h^{\mu_\infty} \mu_\infty + \mathrm{Err} 
	\end{align}
	where $\mathrm{Err} \leq C + \frac{C}{N-1} \sum_{i=2}^N (\g(\tilde \eta_i) + \mcl M_{x_i})$.
\end{proposition}
\begin{proof}
	First, note that for $i \ne j$ in $\{2,\ldots,N\}$, we have
	$$
	\int \g(y - x_j) \nu_{x_i,\tilde \eta_i}(dy) = \g(x_i - x_j)
	$$
	by the definition of $\tilde \eta_i$ and harmonicity of $\g$ away from $0$, and
	$$ 
	\int \g(y - x_i) \nu_{x_i,\tilde \eta_i}(dy) \leq \g(\tilde \eta_i/2)  \leq C \g(\tilde \eta_i).
	$$
	We also have
	\begin{equation} \label{e.alpha.h.bd}
		\int \h^{\mu_\infty}(y) \nu_{x_i,\tilde \eta_i}(dy) = \h^{\mu_\infty}(x_i) + O(\tilde \eta_i^2) \geq \h^{\mu_\infty}(x_i) - C.
	\end{equation}
	since $\h^{\mu_\infty}$ has bounded Laplacian. We are then able to deduce the inequality
	\begin{align}
		\lefteqn{\int_{\R^\d} \Fenergy((y,X_{N,1}),\mu_\infty) \nu(dy) -  \Fenergy(X_{N,1},\mu_\infty) } \quad & \\ \notag & =
		\frac{1}{2(N-1)} \sum_{i,j = 2}^N \int \g(y - x_j) \nu_{x_i,\tilde \eta_i}(dy) - \frac{1}{N-1} \sum_{i = 2}^N \int \h^{\mu_\infty}(y) \nu_{x_i,\tilde \eta_i}(dy) \\ \notag
		&\leq 	\frac{1}{2(N-1)} \sum_{i \ne j = 2}^N  \g(x_i - x_j) - \frac{1}{N-1} \sum_{i=2}^N \h^{\mu_\infty}(x_i) + \mathrm{Err}
		\\ \notag &= \frac{1}{N-1} \(\Fenergy(X_{N,1}, \mu_\infty) - \frac12 \int \h^{\mu_\infty} \mu_\infty\) + \mathrm{Err}.
	\end{align}
	To handle the $\zeta$ term within \eref{alpha.F.bd}, a similar procedure as in \eref{alpha.h.bd} works, using that the Laplacian of $\zeta$ is bounded by $\mcl M_{x_i}$ near $x_i$.
\end{proof}


Applying the squeezing transformation, and using \pref{alpha.F.bd} and Jensen's inequality as usual, yields the following bound on $\rho_1$.
\begin{proposition} \label{p.farfield.squeeze}
	We have
	\begin{equation} \label{e.farfield.squeeze}
		\limsup_{|x| \to \infty} e^{\beta \zeta(x)} \rho_1(x) \leq C e^{-\frac{\beta}{2N} \int \h^{\mu_\infty} \mu_\infty} \E\left[ e^{-\d \log \eta_1 + \beta \mathrm{Err} + \frac{\beta}{N} \(\Fenergy(X_N,\mu_\infty) + \sum_{i=1}^N \zeta(x_i)\)} \right]
	\end{equation}
	for $\mathrm{Err} \leq \frac{C}{N} \sum_{i=1}^N (\g(\eta_i) + \mcl M_{x_i}$).
\end{proposition}
\begin{proof}
	For a measurable set $U \subset \R^\d$, by \pref{split} and $\g \downarrow 0$ at $\infty$, we have
	\begin{equation} \label{e.farfield.d3.1}
		\P(x_1 \in U) \leq \frac{e^{o(1)}}{\Kpart_\beta} \(\int_U e^{-\beta \zeta(y)} dy \) \int_{(\R^\d)^{N-1}} e^{-\beta \Fenergy(X_{N,1},\mu_\infty) - \beta \sum_{i=2}^N \zeta(x_i) } dX_{N,1}
	\end{equation}
	as $\dist(U,0) \to \infty$, where the $o(1)$ term is not necessarily uniform in $\beta$ or $N$, and
	\begin{equation} \label{e.Kpartbeta}
	\Kpart_\beta =  \int_{(\R^\d)^N} e^{-\beta \Fenergy(X_{N},\mu_\infty) - \beta \sum_{i=1}^N \zeta(x_i) } dX_{N}.
	\end{equation}
	Applying \pref{alpha.F.bd}, we compute
	\begin{align*}
		\lefteqn{ e^{-\beta \Fenergy(X_{N,1},\mu_\infty) - \beta \sum_{i=2}^N \zeta(x_i) } } \quad & \\  &\leq \exp\(\frac{-\beta (N-1)}{N}  \int_{\R^\d} \( \Fenergy((y,X_{N,1}),\mu_\infty) + \zeta(y) + \sum_{i=2}^N \zeta(x_i)  \)\nu(dy)  \) \\
		&\quad \times \exp\( \frac{-\beta}{2N} \int \h^{\mu_\infty} \mu_\infty + \beta \mathrm{Err}(X_{N,1})\),
	\end{align*}
	where $\mathrm{Err}$ is as in \pref{alpha.F.bd}. We will integrate the above against $dX_{N,1}$ and apply Jensen's inequality to move the $\nu(dy)$ integral outside the exponential. After renaming $y$ to $x_1$, we have proved
	\begin{align} \label{e.alpha.postJensen}
		\lefteqn{ \int_{(\R^\d)^{N-1}} e^{-\beta \Fenergy(X_{N,1},\mu_\infty) - \beta \sum_{i=2}^N \zeta(x_i) } dX_{N,1} } & \quad \\ &\leq e^{-\frac{\beta}{2N} \int \h^{\mu_\infty} \mu_\infty} \int_{(\R^{\d})^N} e^{\log \nu(x_1) + \beta \mathrm{Err}(X_{N,1}) }e^{-\beta \frac{N-1}{N} \(\Fenergy(X_N,\mu_\infty) + \sum_{i=1}^N \zeta(x_i)\)} dX_N
	\end{align}
	where $\nu(x_1)$ denotes the density of $\nu$ at $x_1$. Note that if $\nu(x_1) > 0$, then there is a unique $i \in \{2,\ldots,N\}$ such that $x_1$ and $x_i$ are mutual nearest neighbors, and 
	$$\nu(x_1) \leq C N^{-1}|x_1 - x_i|^{-\d} \leq CN^{-1} \eta_1^{-\d}.
	$$
	Inserting this inequality into \eref{alpha.postJensen} and \eref{farfield.d3.1} shows
	\begin{align*}
	\P(x_1 \in U) &\leq CN^{-1}e^{-\frac{\beta}{2N} \int \h^{\mu_\infty} \mu_\infty + o(1)} \( \int_U e^{-\beta \zeta(y)} dy \) \\ & \quad  \times \E \left[ e^{-\d \log \eta_1 + \beta \mathrm{Err}(X_{N,1}) + \frac{\beta}{N} \(\Fenergy(X_N,\mu_\infty) + \sum_{i=1}^N \zeta(x_i)\) } \right].
	\end{align*}
	The proposition follows by letting $U = B_r(x)$ and sending $r \to 0$ and $|x| \to \infty$.
\end{proof}

We now seek to bound the terms in the RHS of \eref{farfield.squeeze}. The first observation is that the quantity $\sum_{i=2}^n \g(\tilde \eta_i)$ is well controlled by the energy $\Fenergy(X_N,\mu_\infty)$.

\begin{proposition} \label{p.sum.g.eta}
	We have
	$$
	\sum_{i=1}^N \g(\eta_i) \leq 2 \Fenergy(X_N,\mu_\infty) + CN
	$$
	for a $\| \mu_\infty \|_{L^\infty}$ and $\d$ dependent constant $C$.
\end{proposition}
\begin{proof}
	This is proved in more generality in \cite[Lemma B.2]{AS21}.
\end{proof}

Unfortunately \pref{sum.g.eta} is not strong enough to sufficiently control the $-\d \log \eta_1$ term appearing in \eref{farfield.squeeze}, nor is any estimate not including particle repulsion since the term is infinite for a Poisson point process. We use particular repulsion bounds from the $2$-point function bound in \eref{kpt.comp} (more precisely, we use \eref{kpt.comp.iter}).
\begin{proposition} \label{p.eta1.est}
	We have
	$$
	\E\left[ e^{-\d \log \eta_1 }\right] \leq C_\beta(1+ \max(\log \beta^{-1}, 0)).
	$$
	for a $\beta$-dependent constant $C_\beta$ (depending continuously on $\beta \geq \theta_\ast N^{-2/\d}$).
\end{proposition}
\begin{proof}
	Let $\tilde \P$ denote the Coulomb gas $\P$ conditioned on $x_1$, and let $\tilde \rho_1$ be the corresponding conditional one-point function (as in \pref{1pt.Iso}). Then one has for any $r \leq 1$ that
	$$
	\tilde \P(\eta_1 \leq r/4) \leq\E \left[ \pp X_N(B_r(x_1)) - 1 \ \big |  \ x_1 \right] =  \int_{B_r(x_1)} \tilde \rho_1(x) dx.
	$$
	By \eref{kpt.comp.iter} and some routine estimates, for any $r \leq R/4 \leq \frac18 N^{1/\d}$, we can bound
	\begin{align}
	\tilde \P(\eta_1 \leq r/4)  &\leq CR^{-\d}  e^{\beta (\g(2r) - \g(R/2))}  \int_{B_r(x_1)} e^{C \beta R^2 \tilde {\mcl M}_x}  \int_{B_R(x)}  \tilde \rho_1(z)dzdx  \\ \notag &\leq C\(\frac{r}{R}\)^{\d} e^{C \beta R^2 \mcl M_{x_1} - \beta (\g(2r) - \g(R/2))} \int_{B_{2R}(x_1)} \tilde \rho_1(x)dx.
	\end{align}
	Choosing now $R = \max(1,\beta^{-1/2})$, we find
	\begin{equation} \label{e.eta1.bd1}
	\tilde \P(\eta_1 \leq r/4)  \leq C \(\frac{r}{R}\)^{\d} e^{-\beta \g(2r)} e^{C(1+\beta) \mcl M_{x_1}} \E\left[ \pp X_N(B_{2R}(x_1))  \ \big | \ x_1 \right].
	\end{equation}
	The conditional expectation on the RHS above can be bounded by \tref{t23}. Indeed, the conditional distribution is equivalent to an $N-1$ particle Coulomb gas with a superharmonic perturbation to the potential, so the result still applies. Moreover we can extract a dependence on large $\mcl M_{x_1}$ by globally scaling space by a factor of $\mcl M_{x_1}^{1/\d}$ to see
	\begin{equation} \label{e.t23.app.cond}
	\E\left[ \pp X_N(B_{2R}(x_1))  \ \big | \ x_1 \right] \leq C \max(1,\mcl M_{x_1}) R^\d \leq Ce^{C \mcl M_{x_1}} R^\d,
	\end{equation}
	for a dimensional constant $C$. We now justify this claim in more detail.

	Consider $X_{N,1} = (x_2,\ldots,x_N)$ distributed by $\P$ with conditioned on $x_1$ with $\mcl M_{x_1} \geq 1$, and $Y_{N} =  \mcl M_{x_1}^{1/\d} X_{N}$. Then for fixed $y_1,$ the gas $Y_{N,1} = (y_2,\ldots,y_N)$ has the law of a $N-1$ particle Coulomb gas with new parameters (denoted with bars)
	$$
	\overline \beta = \beta \mcl{M}_{x_1}^{1-\frac{2}{\d}}, \quad \overline V_{N-1}(y) = \mcl M_{x_1}^{\frac{2}{\d} - 1}V_N(\mcl M^{-1/\d}_{x_1} y)
	$$
	with a superharmonic perturbation $\overline{\beta}\sum_{i=2}^N \g(y_i - y_1)$ to the energy.
	This implies $\overline{{\mcl M}}_{y,N-1}= (\mcl M_{x_1,N})^{-1}\mcl M_{\mcl M_{x_1}^{-1/\d}y,N}$. At $y = y_1$, this is $\overline{M}_{y_1,N-1} = 1$, and so \tref{t23} applied to $Y_{N,1}$ conditional on $y_1$ applies with a constant dependent only on $\d$. Looking at the condition \eref{K.cond.1} with the new parameters $\overline{R} = \mcl M_{x_1}^{1/\d}R$ and $\overline{Q} =Q$, we find that
	\begin{equation} \label{e.Qcond}
	\overline{Q} \geq C \overline{R}^\d + C \overline{\beta}^{-1} \overline{R}^{\d-2} \iff Q \geq C \mcl M_{x_1} R^\d + C \beta^{-1} R^{\d-2},
	\end{equation}
	where now the constant $C$ is dimensional. Substituting $R = \max(1,\beta^{-1/2})$, one can easily verify
	$$
	\P(\pp X_N(B_{2R}(x_1)) \geq Q+1 \ | \ x_1) \leq e^{-C^{-1} \beta R^{-\d+2} Q^2} \leq e^{-\mcl M_{x_1} Q}
	$$
	for all $Q \geq C \mcl M_{x_1} R^d$, which justifies \eref{t23.app.cond}.

	Combining \eref{eta1.bd1}, \eref{t23.app.cond}, and taking expectation yields
	$$
	 \P(\eta_1 \leq r/4) \leq C r^{\d} e^{\beta \g(2r)} \E\left[ e^{C(1+\beta) \mcl M_{x_1}} \right] \leq C_\beta r^{\d} e^{\beta \g(2r)},
	$$
	where the last bound follows from \lref{Mx.bd} below. Taking $C = C_\beta$ below, we deduce
	\begin{align} \label{e.deta.final}
	\E \left[ e^{-\d \log \eta_1} \right] &= 4^\d + \d\int_0^{\frac14} r^{-\d - 1} \P(\eta_1 \leq r) dr \\ \notag
	&\leq 4^\d + C\int_0^{\frac14} r^{-1}e^{\frac{-\beta}{8^{\d-2}} r^{-d +2} } dr \leq 4^\d + C\int_0^{\frac{1}{32}} r^{-1}e^{-\beta r^{-d +2} } dr.
	\end{align}
	One can estimate
	$$
	\int_0^{\beta^{1/(d-2)}}  \frac{e^{-\beta r^{-(d-2)}}}{r}dr \leq  \beta \int_0^{\beta^{1/(d-2)}}  \frac{e^{-\beta r^{-(d-2)}}}{r^{\d-1}}dr \leq C,
	$$
	where the last inequality follows from the change of variables $u = r^{-(\d-2)}$, and if $\beta^{1/(d-2)} < 1/{32}$, one estimates the remainder of the integral by
	$$
	\int_{\beta^{1/(d-2)}}^{\frac1{32}}  \frac{e^{-\beta r^{-(d-2)}}}{r}dr \leq 	\int_{\beta^{1/(d-2)}}^{\frac1{32}}  \frac{dr}{r} \leq C + C \log \beta^{-1}.
	$$
	Inserting these bounds into \eref{deta.final} finishes the proof.
\end{proof}

Finally, we fill in the needed bound for \pref{eta1.est}. The result is a relatively straightforward consequence of \eref{farfield.reduction3} and our assumptions on $V$. The argument in \lref{Mx.bd} is similar to the argument in \cite[Theorem 1.12]{CHM18}.
\begin{lemma} \label{l.Mx.bd}
	Let $\d \geq 3$. For any $t > 0$, we have
	$$
	\sup_{N \geq 1} \E \left[ e^{t \mcl M_{x_1}} \right] \leq C_t
	$$
	for a $V_1, t$ dependent constant $C_t$.
\end{lemma}

\begin{proof}
	Note that
	\begin{equation}\label{e.zeta.V}
	\zeta(x) = \h^{\mu_\infty}(x) + V(x) + O(N^{2/\d}) \geq V(x) + O(N^{2/\d}).
	\end{equation}
	By \eref{farfield.reduction3}, and the bound on the density of the $x_1$ marginal it implies, we have for any measurable set $U$ that
	\begin{align}
		\E\left[ e^{t \mcl M_{x_1}} \1_{x_1 \in U} \right] &\leq N^{-1} e^{C \beta N^{2/\d}} \int_{U} e^{-\beta \zeta(y) + t \mcl M_y} dy \\ \notag
		&\leq N^{-1} e^{C \beta N^{2/\d}} \sup_{z \in U} e^{-(\beta -  N^{-2/\d}\theta_\ast/2)V(z) + t \mcl M_z}  \int_{\R^\d} e^{-\frac{N^{-2/\d}\theta_\ast}{2} V(y)}dy.
	\end{align}
	The integral on the RHS is bounded by $Ne^{C\beta N^{2/\d}}$ by \eref{V.assume}. Indicating the $N$-dependence in what follows and denoting $z' = N^{-1/\d}z$, note that
	$$
	-\(\beta_N - \frac{N^{-2/\d}\theta_\ast}{2}\) V_{N}(z) + t \mcl M_{z,N} = -\(N^{2/\d} \beta_N- \frac{\theta_\ast}{2}\) V_1(z') + t \mcl M_{z',1}.
	$$
	For any $K > 0$, we have for all $|z'|$ large enough by \eref{VM.assume} that
	$$
	\(N^{2/\d} \beta_N- \frac{\theta_\ast}{2}\) V_1(z') \geq t \mcl M_{z',1} + K \theta_\ast
	$$
	Indeed, this follows from $\frac{\theta_\ast}{2} V_1(z') \geq 2t \mcl M_{z',1}$ and $V_1(z') \to \infty$ as $|z'| \to \infty$. Choosing $K$ large enough then shows that, for $U = (B_{N^{1/\d} R}(0))^c$ for a large enough $R$, we have
	$$
	\E\left[ e^{t \mcl M_{x_1,N}} \1_{x_1 \in U} \right]  \leq 1.
	$$
	Since $\mcl M_{x,N}$ is bounded for $x \in U^c$, we achieve the desired result.
\end{proof}

\begin{proof}[Proof of \pref{farfield.simple} in $\d \geq 3$]
	Consider the error term
	$$
	\mathrm{Err}(X_{N,1}) = \frac{C}{N} \sum_{i=1}^N \g(\eta_i) + \mcl M_{x_i}
	$$
	 from \pref{farfield.squeeze}. By \eref{zeta.V} and \eref{VM.assume} (and the $N$ scaling), we have
	 $$
	 \mcl M_{x,N} \leq CV_1(N^{-1/\d}x) + C = CN^{-2/\d} V(x) + C \leq C \zeta(x) + C,
	 $$
	 and so \pref{sum.g.eta} implies
	 $$
	 \mathrm{Err}(X_{N,1}) \leq \frac{C}{N} \( \Fenergy(X_N,\mu_\infty) + \sum_{i=1}^N \zeta(x_i) \) + C.
	 $$
	Therefore, by \pref{farfield.squeeze} there exists $C_0 \geq 1$ such that
	\begin{equation} \label{e.limrep.twoexp}
		\limsup_{|x| \to \infty} e^{\beta \zeta(x)} \rho_1(x) \leq C e^{-\frac{\beta}{2N} \int \h^{\mu_\infty} \mu_\infty} \E\left[ e^{-\d \log \eta_1 + \frac{C_0\beta}{N} \(\Fenergy(X_N,\mu_\infty) + \sum_{i=1}^N \zeta(x_i)\)} \right].
	\end{equation}
	The expectation in the RHS of \eref{limrep.twoexp} can be written and bounded as
	\begin{equation} \label{e.two.expectations}
		\tilde \E\left[e^{-\d \log \eta_1} \right] \frac{\Kpart_{\tilde \beta}}{\Kpart_\beta}\leq C\max(\log \beta^{-1},1) \frac{\Kpart_{\tilde \beta}}{\Kpart_\beta}
	\end{equation}
	where $\tilde \E$ denotes expectation with respect to the gas $\P^V_{N, \tilde \beta}$ with $\tilde \beta = (1 - \frac{C_0}{N}) \beta$ and $\Kpart_{\beta}$ is as in \eref{Kpartbeta}. We used \pref{eta1.est} to bound the $\tilde \E$ expectation term.
	
	Jensen's inequality implies that for all $N$ large enough, we have
	\begin{equation} \label{e.Kratio.close}
	\frac{\Kpart_{\tilde \beta}}{\Kpart_\beta} = \E \left[  e^{\frac{C_0\beta}{N} \(\Fenergy(X_N,\mu_\infty) + \sum_{i=1}^N \zeta(x_i)\) } \right] \leq   \E \left[ e^{\frac{\beta}{2} \(\Fenergy(X_N,\mu_\infty) + \sum_{i=1}^N \zeta(x_i)\)} \right]^{\frac{2 C_0}{N}} = \(\frac{\Kpart_{\beta/2}}{\Kpart_\beta} \)^{\frac{2C_0}{N}}.
	\end{equation}
	
	We claim that
	$$
	\log \Kpart_{\beta} = N \log N + O(N (1+\beta)).
	$$
	We could not find a reference inclusive of small $\beta$ and our exact form of $\Kpart_\beta$, so we sketch an essentially well-known argument here. First, the inequalities
	$$
	\Fenergy(X_N,\mu_\infty ) \geq -C N, \quad \int_{\R^\d} e^{-\beta \zeta(x)} dx \leq CN,
	$$
	the first coming from a standard truncation argument, as in \cite[Corollary 4.14]{S24} or \cite[Lemma 3.7]{AS21}, and the second coming from $e^{-\beta \zeta(x)}dx \leq C \mu_\theta(dx)$, immediately imply that $$\log \Kpart_\beta \leq C N \beta + N \log N.$$ For the reverse inequality, beginning with the inequality
	$$
	dX_N \geq \frac{1}{\| \mu_\infty \|_{L^\infty}^N} \mu_\infty^{\otimes N}(dX_N) \geq e^{-CN} \mu_\infty^{\otimes N}(dX_N),
	$$
	we find via Jensen's inequality and a short computation that
	\begin{align*}
	\Kpart_{\beta} &\geq e^{-CN} \int_{(\R^\d)^N} e^{-\beta \Fenergy(X_N,\mu_\infty)} \mu_\infty^{\otimes N}(dX_N) \\ &\geq e^{N \log N - CN} \exp\( -\beta N^{-N}  \int_{(\R^\d)^N}  \Fenergy(X_N,\mu_\infty) \mu_\infty^{\otimes N}(dX_N) \) \\
	&= e^{N \log N - CN} \exp\( \frac{\beta}{2N} \iint_{\R^\d \times \R^\d} \g(x-y) \mu_\infty(dx) \mu_\infty(dy)\) \geq e^{ N \log N - CN}.
	\end{align*}

	Using our $\Kpart_\beta$ estimate, \eref{Kratio.close}, \eref{two.expectations}, and \eref{limrep.twoexp}, we find
	$$
	\limsup_{|x| \to \infty} e^{\beta \zeta(x)} \rho_1(x)  \leq C^{1+\beta} \max(\log \beta^{-1},1) e^{-\frac{\beta}{2N} \int \h^{\mu_\infty} \mu_\infty}.
	$$
	We conclude by noting the scaling inequality
	$$
	\frac{\beta}{2N} \int \h^{\mu_\infty} \mu_\infty \geq  \beta c_{V_1} N^{2/\d}
	$$
	for some $c_{V_1} > 0$. 
\end{proof}

	\addcontentsline{toc}{section}{References}
	\bibliographystyle{alpha}
	\bibliography{confinementbib}{}
	
	\vskip .5cm
	\noindent
	\textsc{Eric Thoma}\\
	Department of Mathematics, Stanford University. \\
	Email: {thoma@stanford.edu}.
	\vspace{.2cm}

\end{document}